\documentclass[a4paper,12pt,reqno]{article}
\usepackage[english]{babel}
\usepackage{amsmath,amsfonts,amssymb,amsthm,bbm}
\usepackage{graphics,epsfig,psfrag} 
\usepackage{paralist,subfigure,multirow,float}
\usepackage{hyperref} 
\usepackage[active]{srcltx} 
\usepackage[dvipsnames]{xcolor}
\usepackage{soul} 
\usepackage[latin1]{inputenc}
\usepackage[OT1]{fontenc}

\newtheorem{theorem}{Theorem}[section]
\newtheorem{corollary}[theorem]{Corollary}
\newtheorem{lemma}[theorem]{Lemma}
\newtheorem{proposition}[theorem]{Proposition}
\newtheorem{remark}[theorem]{Remark}

\newtheorem{definition}[theorem]{Definition}

\theoremstyle{definition}

 \usepackage{autonum} 

\def\R{\mathbb{R}}

\def\epsilon{\varepsilon}

\let\.=\cdot
\let\0=\emptyset

\let\mc=\mathcal

\def\1{\mathbbm{1}}

\def\hat{\widehat}
\def\tilde{\widetilde}

\newenvironment{formula}[1]{\begin{equation}\label{#1}}{\end{equation}\noindent}

\numberwithin{equation}{section}
\newcommand{\be}{\begin{equation}}
\newcommand{\ee}{\end{equation}}
\newcommand{\baa}{\begin{array}}
\newcommand{\eaa}{\end{array}}
\newcommand{\ba}{\begin{eqnarray}}
\newcommand{\ea}{\end{eqnarray}}
\def\Fi#1{\begin{formula}{#1}}
\def\Ff{\end{formula}\noindent}
\def\BS{\color{Bittersweet}}

\def\W{\mc{W}}

\begin{document}
\date{}
\title{\bf{Spreading dynamics for the Lotka-Volterra system with general initial supports: the strong competition}}
\author{Hongjun Guo$^{\dagger}$\thanks{Corresponding author. Email: guohj29@tongji.edu.cn}\\
\\
\footnotesize{$^{\dagger}$School of Mathematical Sciences, Key Laboratory of Intelligent Computing and Applications }\\
\footnotesize{(Ministry of Education), Institute for Advanced Study, Tongji University, Shanghai, China}}
\maketitle
	
\begin{abstract}
\noindent{This paper investigates the high-dimensional spreading dynamics of the Lotka-Volterra strong competition-diffusion system defined on $\mathbb{R}^N$ with $N\ge 2$, where two competing species are initially supported on disjoint measurable subsets $U,\, V\subset\mathbb{R}^N$ that may be unbounded. Under the strong competition regime $a,b>1$, the system possesses two stable homogeneous equilibria $(1,0)$ and $(0,1)$, and exhibits layered wave structures known as propagating terraces along distinct spatial directions, whose rigorous characterization remains largely open in multi-dimensional settings. Most existing analytical results for propagating terraces are restricted to one-dimensional domains or compactly supported initial data. In this work, we introduce projection paths and star-shaped geometric assumptions to analyze the anisotropic propagation of two competing species. For each unit direction $e\in\mathbb{S}^{N-1}$, we derive a dichotomy of long-time spreading behavior governed by the comparison between the single-species spreading speeds $w_u(\xi)$ and $w_v(\xi)$ along the projection path associated with $e$: if $w_u(\xi)>w_v(\xi)$ for all relevant directions $\xi$ on the path, species $u$ invades the whole space and $v$ goes extinct in direction $e$; otherwise, a propagating terrace emerges, with $u$ advancing at the interspecific bistable wave speed $c_{uv}(e)$ inside the region occupied by $v$, which spreads at its own single-species speed $w_v(e)$. }
\end{abstract}
	


\section{Introduction and main results}\label{intro}

This paper is devoted to the study of propagation phenomenon for the following Lotka-Volterra competition-diffusion system in $\mathbb{R}^{N}$ with $N\ge 2$
\begin{equation}\label{pro1}
	\begin{cases}
		u_{t}=d\Delta u+ru(1-u-av), &t>0,\, x\in \mathbb{R}^{N} ,\\
		v_{t}=\Delta v+v(1-v-bu) ,  &t>0,\, x\in \mathbb{R}^{N}.
	\end{cases}
\end{equation}
In the system, $u(t, x)$, $v(t, x)$ represent the population densities of two competing species at the  time $t$ and the space position $x$ respectively, $d$ stands for the diffusion rate of $u$, $r$ represents the intrinsic growth rate of $u$, $a$ and $b$ are the competition coefficients for two species respectively. All the  parameters are  positive. The two species are set to initially occupy two distinct domains in the space respectively. We are interested in the spreading dynamics of these two species. Mathematically, the initial values are set to satisfy
$$u(0,x)=u_0(x)=\1_{U} \quad \hbox{and} \quad v(0,x)=v_0(x)=\1_{V},$$
where $U$ and $V$ are measurable subsets of $\mathbb{R}^N$ (which can be unbounded in general) satisfying $U\cap V=\emptyset$ and they are called initial support. Here, $\1_{U}$ is the characteristic function of the set $U$, that is,
\begin{eqnarray*}
\1_{U}=\left\{\begin{array}{lll}
1, &&\hbox{for $x\in U$},\\
0, &&\hbox{for $x\in \R^N\setminus U$}.
\end{array}
\right.
\end{eqnarray*}
The initial condition could be more general, such as multiples $\alpha \1_{U}$ and $\beta \1_{V}$ of characteristic functions, with $\alpha>0$ and $\beta>0$. But for presenting our results as clear as possible, we only consider $\1_{U}$. The solution of \eqref{pro1} is understood in the sense that $(u,v)(t,\cdot)\rightarrow (u_0,v_0)$ as $t\rightarrow 0^+$ in $L_{loc}^1(\R^N)\times L_{loc}^1(\R^N)$.

The spreading dynamics of two competing species under staged invasions has long been a core open subject in the theory of competitive reaction-diffusion systems. The pioneering investigation by Shigesada and Kawasaki \cite{SK} in 1997 first estimated the spreading speeds of two competitors based on the linear determinacy and local determinacy conjectures, yet they themselves acknowledged that these conjectures might require revision, as Hosono's numerical simulations \cite{Ho} demonstrated scenarios where linear determinacy breaks down. 
Whether linear determinacy holds depends sensitively on system parameters: it is valid for certain parameter regimes verified by Lewis, Li, Weinberger \cite{LLW}, Huang \cite{Hu}, and Alhasanat, Ou \cite{AO}, while it fails in other regimes documented by Huang, Han \cite{HH} and Alhasanat, Ou \cite{AO}. To date, a sharp necessary and sufficient parameter condition that fully distinguishes linear determinacy from its failure remains undiscovered, standing as one of the most prominent unsolved problems for the Lotka-Volterra competition-diffusion system. Closely intertwined with this issue is the dichotomy between pulled fronts (driven by exponentially decaying tails) and pushed fronts (propelled by the wave's rear), alongside the partially verified conjecture of Roques, Hosono, Bonnefon and Boivin \cite{RHBB}, which claims an exact equivalence between pulled/pushed front classification and the validity/failure of linear determinacy.

The concept of \textit{propagating terraces} (composite structures consisting of sequential traveling fronts with distinct speeds) adds further complexity to this problem. Coined by Ducrot, Giletti and Matano \cite{DGM} (building on the earlier term ``minimal decomposition'' from Fife and McLeod \cite{FM} for homogeneous scalar equations), propagating terraces describe dynamics where two invasion processes spread at separate velocities. For system \eqref{pro1}, the two distinct spreading speeds originate respectively from the monostable Fisher-KPP scalar equation and a bistable differential system analyzed by Kan-On \cite{Kan}. The coexistence of two disparate spreading speeds forming a propagating terrace rules out direct deployment of most classical scalar PDE techniques.

Against the backdrop of these long-standing open questions, the present paper addresses the high-dimensional spreading dynamics of the competitive system \eqref{pro1} under the strong competition regime $a,b>1$, where the system admits bistable competition dynamics. Unlike one-dimensional frameworks studied in most existing literature, propagation in $\mathbb{R}^N$ ($N\ge 2$) exhibits prominent anisotropic behavior: \textit{propagating terraces} may emerge along different spatial directions, with direction-dependent spreading speeds yet to be rigorously characterized. Existing rigorous treatments of propagating terraces are largely confined to scalar one-dimensional equations or simple one-dimensional systems with Heaviside-type/compactly supported initial data, leaving high-dimensional anisotropic terrace spreading in bistable competitive systems largely unexplored. Our main target is to rigorously identify and compute the exact spreading speed corresponding to each spatial direction, filling a critical gap in the theory of multi-species invasion in high-dimensional media.

\subsection{Known results for a single species}

Even if there is only one species, the spreading dynamics is not a trivial extension for general initial supports from compact initial supports. Suppose $v\equiv 0$ and the system then reduces into the following Fisher-KPP equation (which is named by the pioneering works of Fisher \cite{F} and Kolmogorov, Petrovsky and Piskunov \cite{KPP})
\be\label{eq:u}
u_{t}=d\Delta u+ru(1-u), \quad t>0,\, x\in\R^N,
\ee
with the initial value $u_0(x)=\1_U$. One knows from the hair trigger effect of Aronson and Weinberger \cite{Aro} that if the Lebesgue measure of $U$ is positive, $u(t,x)\rightarrow 1$ locally uniformly in $\R^N$ as $t\rightarrow +\infty$. The intriguing problem here is how to describe the expanding region where $u$ is close to $1$. Key notions are the spreading speed and the spreading set. 

\begin{definition}
For a given unit vector $e\in\mathbb{S}^{N-1}$, a quantity $\omega_u(e)$ is called spreading speed of $u$ if
\ba\label{conv:ss}
\left\{\begin{array}{lll}
u(t,cte)\rightarrow 1 &&\hbox{as $t\rightarrow +\infty$ for every $0\le c<\omega_u(e)$},\\
u(t,ct e)\rightarrow 0 &&\hbox{as $t\rightarrow +\infty$ for every $c>\omega_u(e)$}.
\end{array}
\right.
\ea
\end{definition}

\begin{definition}\label{def:W}
A set $\W_u\subset\R^N$ is called spreading set of $u$  if $\W_u$ coincides  with the interior of its closure and satisfies 
\be\label{subset}
\lim_{t\to+\infty}\,\Big(\min_{x\in C}u(t, tx)\Big)=1, \, \text{ for any non-empty compact set }C\subset \mc{W}_u,
\ee
\be\label{superset}
\lim_{t\to+\infty}\,\Big(\max_{x\in C}u(t,tx)\Big)=0,  \, \text{ for any non-empty compact set }C\subset\R^N\setminus \overline{\mc{W}_u}.
\ee
If only \eqref{subset} (resp. \eqref{superset}) holds, we say that $\mc{W}_u$ is a spreading subset (resp. superset) of $u$.
\end{definition}

The requirement that the spreading set coincides with the interior of its closure immediately implies its uniqueness. The spreading speed $\omega_v(e)$ and set $\W_v$ can be defined for $v$ in the same way and we are not going to repeat it.

When the initial support $U$ is a compact set with a positive Lebesgue measure, the classical results of Aronson and Weinberger~\cite{Aro} tell us that the spreading speed and set of the solution $u$ for \eqref{eq:u} are  $\omega_u(e)=c_u:=2\sqrt{dr}$ for all $e\in\mathbb{S}^{N-1}$ and $\W_u=B_{c_u}$ respectively.\footnote{The notation $B_r(x_0)$ denotes the ball with center $x_0$ and radius $r$. For  $x_0=0$, denote the ball by $B_r$ for short.} This means that the species $u$ spreads with the constant speed $c_u$ in every direction, which is independent of the shape of the set $U$. Here,  $c_u=2\sqrt{d r}$ is the minimal speed of traveling fronts such that the equation \eqref{eq:u} admits a traveling front $\phi(x\cdot e-ct)$ for $t\in\mathbb{R}$, $x\in\mathbb{R}^N$ and every $e\in\mathbb{S}^{N-1}$ if and only if $c\ge c_u$. The profile $\phi(\xi)$ and the speed $c$ of a traveling front satisfy the following equation
\begin{equation}\label{into1}
	\begin{cases}
		d\phi''+c\phi'+r \phi(1-\phi)=0,  \, ~\text{in}~\xi \in \mathbb{R}
		\\
		\phi(-\infty)=1,\, \phi(+\infty)=0.
	\end{cases}
\end{equation}
The  profile $\phi$ is unique up to shifts for every $c\ge c_u$.
In the sequel, we denote the traveling front with the minimal speed of \eqref{eq:u} by $\phi(x\cdot e- c_{u}t)$ with $\phi(0)=1/2$. 

There is a vast literature investigating the large time dynamics of solutions of \eqref{eq:u} for various types of $f$ with initial conditions $u_0$ that are compactly supported or converge to $0$ at infinity. For extinction or invasion results upon the size and/or the amplitude of the initial condition $u_0$, we refer to \cite{Aro, DM, LK, MZ, MZ2, Xin}, and for general local convergence and quasiconvergence results at large time, we refer to \cite{DM, DP, MP1, MP2, P}. 

When the initial support $U$ is a general measurable set, there is much less investigation. In this case, the spreading of $u$ keeps  a memory of the set $U$, see the work of Hamel and Rossi \cite{HR}. To describe the spreading dynamics, they introduced geometrical notions called bounded direction and unbounded direction defined as follows.

\begin{definition}[\cite{HR}]\label{deuv}
The set of bounded directions of $U$ and the set of unbounded directions of $U$ are given by	
	\begin{equation}
		\mathcal{B}(U):=\left\lbrace 
		\xi\in \mathbb{S}^{N-1}:\liminf_{\tau\to
			+\infty}\frac{ \text{{\rm dist}}
			(\tau \xi , U)}{\tau}>0\right\rbrace,
	\end{equation}
	and 
	\begin{equation}
		\mathcal{U}(U):=\left\lbrace 
		\xi\in \mathbb{S}^{N-1}:\lim_{\tau\to
			+\infty}\frac{\text{{\rm dist}}
			(\tau \xi , U)}{\tau}=0 \right\rbrace,
	\end{equation}
	respectively.
\end{definition}

By the definition, sets $\mathcal{B}(U)$ and $\mathcal{U}(U)$ are relatively open and closed in $\mathbb{S}^{N-1}$ respectively and the following geometrical properties hold:
\begin{itemize}
	\item $\xi \in \mathcal{B}(U)$ if and only if there exists an open cone $\mathcal{C} \supset \mathbb{R}^{+}\xi$ such that $U \cap \mathcal{C} \subset B_R$ for some large $R>0$, i.e., $\mathbb{R}^{+}\xi \subset \mathcal{C} \subset (\mathbb{R}^N \setminus U) \cup B_R$.
	\item  $\xi \in \mathcal{U}(U)$ means for any open cone $\mathcal{C} \supset \mathbb{R}^{+}\xi\), \(U \cap \mathcal{C}$ is unbounded.
\end{itemize}
 For instance,  when $U$ is an open cone with vertex $0$, then the set of unbounded directions 
 is $\overline{U} \cap \mathbb{S}^{N-1}$ and the set of bounded directions is $\mathbb{S}^{N-1}\setminus \overline{U}$. Another notion needed is
positive-distance-interior $U_{\rho}$ (with $\rho>0$) of the set $U$ defined as
\[
U_{\rho}:=\left\lbrace x\in U: 
\text{dist}(x, \partial U)\ge\rho \right\rbrace. 
\]

By these notions, Hamel and Rossi \cite{HR} proved a variational formula of the spreading speed of the solution $u$ for \eqref{eq:u} (they dealt with a more general nonlinearity which covers $ru(1-u)$). Precisely, if there is $\rho>0$\footnote{In \cite{HR}, they require the positive-distance-interior $U_{\rho}$ to be nonempty where $\rho>0$ is a constant such that if $U=B_{\rho}$, then $u\rightarrow 1$ locally uniformly in $\R^N$ as $t\rightarrow +\infty$. However, by the hair trigger effect, $\rho$ can be any positive constant in the Fisher-KPP case.} such that $U_{\rho}\neq \emptyset$ and $U$ satisfies
\be\label{BUUU}
\mathcal{B}(U)\cup	\mathcal{U}(U_\rho)= \mathbb{S}^{N-1},
\ee
then for every $e\in\mathbb{S}^{N-1}$, the spreading speed $w_u(e)$ of $u$ is given by
\be\label{omega-u}
	w_u(e)=\sup_{\xi\in \mathcal{U}
	(U), ~\xi\cdot e\ge 0}
\frac{c_u}{\sqrt{1-(\xi\cdot e)^{2}}},
\ee
with the conventions: $w_u(e)=c_u$ if there is no $\xi \in \mathcal{U}(U)$ such that $\xi \cdot e\ge 0$, and $w_u(e)=+\infty$ if $e\in \mathcal{U}(U)$. If $\mathcal{U}(U)\neq \emptyset$, the above formula can be expressed in a more geometrical way
$$w_u(e)=\frac{c_u}{\hbox{dist}(e,\R^+\mathcal{U}(u))}.$$
Thus, for a direction $e$ such that $w_u(e)>c_u$, there is $\xi\in\mathcal{U}(U)$ such that $c_u=w_u(e)\sqrt{1-(\xi\cdot e)^2}$. Then,  the species at $w_u(e)e$ can be regard as coming from $(w_u(e)\xi\cdot e)\xi$ initially after a unit time. The speed on the path from $(w_u(e)\xi\cdot e)\xi$ to $w_u(e)e$ is $c_u$. If $w_u(e)=c_u$, the species at $w_u(e)e$ can be regard as coming from $0$ with speed $c_u$. Clearly, from the formula \eqref{omega-u}, the map $e\mapsto w_u(e)\in [c_u,+\infty)$ is continuous in $\mathbb{S}^{N-1}$. As far as the uniformity of the convergence \eqref{conv:ss} with respect to $e\in\mathbb{S}^{N-1}$ is concerned, they also proved that under the same conditions, the spreading set $\mathcal{W}_u$ of $u$ is given by 
\be\label{W-u}
\mathcal{W}_u=\left\lbrace re: 0\le r<w_u(e) \right\rbrace =\mathbb{R}^{+}\mathcal{U}(U)+B_{c_u}.
\ee
Its boundary $\partial\W_u$ is continuous by the continuity of $w_u(e)$.
For more discussions about the condition~\eqref{BUUU} and formulae~\eqref{omega-u}-\eqref{W-u}, we refer to \cite{HR}.
 
In absence of the species $u$, $c_{v}:=2$ is the minimal speed associated with  the equation $v_t=\Delta v+v(1-v)$ and let $\psi(x\cdot e- c_{v}t)$ be the traveling front with the minimal speed and $\psi(0)=1/2$. Then, equipped with the initial value $v(0,x)=v_0(x)=\mathbbm{1}_V$, there are similar formulae for $v$ as \eqref{omega-u}-\eqref{W-u}. We write down as follows:  if there is $\rho>0$ such that $V_{\rho}\neq \emptyset$ and $V$ satisfies
\be\label{BVUV}
\mathcal{B}(V)\cup	\mathcal{U}(V_\rho)= \mathbb{S}^{N-1},
\ee
then for every $e\in\mathbb{S}^{N-1}$, the spreading speed $w_v(e)$ of $v$ is given by
\be\label{omega-v}
	w_v(e)=\sup_{\xi\in \mathcal{U}
	(V), ~\xi\cdot e\ge 0}
\frac{c_v}{\sqrt{1-(\xi\cdot e)^{2}}},
\ee
with the conventions: $w_v(e)=c_v$ if there is no $\xi \in \mathcal{U}(V)$ such that $\xi \cdot e\ge 0$, and $w_v(e)=+\infty$ if $e\in \mathcal{U}(V)$, and the spreading set $\mathcal{W}_v$ of $v$ is given by 
\be\label{W-v}
\mathcal{W}_v=\left\lbrace re: 0\le r<w_v(e) \right\rbrace =\mathbb{R}^{+}\mathcal{U}(V)+B_{c_v}.
\ee
The map $e\mapsto w_v(e)\in [c_v,+\infty)$ is continuous in $\mathbb{S}^{N-1}$ and hence, $\partial \mathcal{W}_v$ is also continuous. 

\subsection{Interaction of two species}

By the definition of the spreading set and above formulae for single species, one can image that the species $u$ spreads almost as $t\W_u$ at large time in the absence of species $v$ (that is, $V=\emptyset$) and  the species $v$ spreads almost as $t\W_v$ at large time in the absence of species $u$ (that is, $U=\emptyset$). The problem arising here is how they spread when both species appear in the space. 

For simplicity of presentation, we only consider the strong competition case in this paper, that is,
\begin{itemize}
\item[\textbf{(A1)}] (strong competition) the competition coeffcients $a$, $b>1$.
\end{itemize}
Part of our results are independent of the choice of $a$ and $b$. For example, Lemmas~\ref{lemma:bar-e=0}-\ref{lemma:eball-c}, \ref{lemma:eball-v1}, \ref{lemma:eball-v} do not need the assumption (A1).
In the strong competition case, the system \eqref{pro1} has two stable equilibria $(0,1)$ and $(1,0)$.
One typical tool to describe the interaction of two species and their spreading properties is the traveling front of the system \eqref{pro1}. For the one-dimensional case of \eqref{pro1} ($N=1$), there exists a traveling front connecting $(0, 1)$ and $(1, 0)$, see Gardner \cite{GE} and Kan-On \cite{Kan}, which can be trivially extended to high-dimensional spaces. Precisely, for the high-dimensional case ($N\geq 2$), there is a unique (up to translations) traveling front of (\ref{pro1}) in the form of $(u, v)(t,x)=(\Phi,  \Psi)(x\cdot e-c_{uv}t)$ for every $e\in\mathbb{S}^{N-1}$. That is, there exists a unique speed  $c_{uv}$ satisfies 
$$c_{uv}\in (-2, 2\sqrt{dr})$$
 such that  the profiles $\Phi(\xi)$, $\Psi(\xi)$ satisfy
\begin{equation}\label{pro2}
	\begin{cases}
		d\Phi''+c_{uv}\Phi'+
		r\Phi(1-\Phi-a\Psi)=0,~ &\xi \in \mathbb{R},\\
		\Psi''+c_{uv}\Psi'+\Psi(1-\Psi-b\Phi)=0,~
		&\xi \in \mathbb{R},\\
		\left( \Phi, ~\Psi\right)(-\infty)=(1, 0), ~ 
		\left( \Phi, ~\Psi\right)(+\infty)=(0, 1),\\
		\Phi'<0, ~\Psi'>0,~&\xi \in \mathbb{R}.
	\end{cases}
\end{equation}
We fix the traveling front by $\Phi(0)=1/2$.
 More properties for the profiles $\Phi$, $\Psi$ can be referred to \cite{Kan,KF}. The fact $-c_v=-2<c_{uv}<c_u=2\sqrt{dr}$ implies that the competition of two species slows down the spreading speed of a single species.
 
In the one-dimensional case, Carr\`{e}re \cite{Car} considered the initial value problem when the two strong competition species are initially absent from the right half-line $x > 0$, and the slow one dominates the fast one on $x < 0$. She found that the fast one will invade the right space at its Fisher-KPP speed, and the slow one will be replaced by or will invade the fast one at the speed $c_{uv}$.  Then, Peng, Wu and Zhou in \cite{PWZ} improved the results of  Carr\`{e}re \cite{Car} and derived the sharp estimates of the 
spreading speeds based on the technology of \cite{Ham}. It is worth mentioning that the spreading dynamics is more complicated in the strong-weak competition case (that is, $a<1<b$) and the weak competition case (that is, $a,\, b<1$), even if the same problem as \cite{Car} was addressed, see \cite{GL,LL,LLL}. Also refer to \cite{AX} for the critical competition case (that is, $a=b=1$).

Throughout this paper, we also assume that 
\begin{itemize}
	\item[\textbf{(A2)}] ~the speed $c_{uv}>0$.
\end{itemize}
This assumption means that if we only observe the traveling front $(\Phi,  \Psi)(x\cdot e-c_{uv}t)$, the species $u$ is the winner in the competition since its territory, namely the region where $(u,v)$ is close to $(1,0)$, is expending at the speed $c_{uv}$ in the direction $e$ as $t\rightarrow +\infty$. Some sufficient conditions on parameters to ensure $c_{uv}>0$ can be referred to \cite{GLi,Kan} and we also refer to \cite{Gir,GN,MZO,RM} for some related discussions. For the case $c_{uv}<0$, that is, the species $v$ is the winner in the competition, one can get similar results through our analysis.

In our previous work \cite{BG}, we extended Carr\`{e}re's results to the high dimensional case for some specific initial conditions. In particular, we proved that when $U$ is large, then the species $u$ will spread successfully. 

\begin{proposition}[\cite{BG}]\label{proposition1}
Assume that {\rm \textbf{(A1)}-\textbf{(A2)}} hold.	There is $\rho>0$ such that the solution $(u,v)(t,x)$ of \eqref{pro1} associated with $U=B_\rho$ and $V$ being any set with $U\cap V=\emptyset$ satisfies 
	$$(u,v)(t,x)\rightarrow (1,0), \hbox{ locally uniformly in $\mathbb{R}^N$ as $t\rightarrow +\infty$}.$$
\end{proposition}

Once the species $u$ spreads successfully, the spreading speed of $u$ is determined by the comparison of the speeds of single species in different directions. When $U$ and $V$ are compact and $U\cap V=\emptyset$, remember that the speeds of single species $u$ and $v$ are $c_u$ and $c_v$ in all directions respectively and we have the following results in \cite{BG}. If $c_u>c_v$, then $v$ extincts and $u$ spreads at speed $c_u$, that is,
\begin{eqnarray*}
\left\{\begin{array}{lll}
			\displaystyle\lim_{t\to +\infty}\sup_{|x|\le ct}
			|u(t, x)-1| =0, &&\text{for}~
			0<c<c_{u},\\
			\displaystyle\lim_{t\to +\infty}\sup_{|x|\ge ct}
			|u(t, x)| =0, && \text{for}~
			c>c_{u},\\
			\displaystyle\lim_{t\to +\infty}\sup_{x\in\mathbb{R}^N} |v(t, x)|=0.&&
		\end{array}
		\right.
	\end{eqnarray*}	
If $c_u<c_v$,	then $v$ spreads at speed $c_v$ and $u$ invades $v$ at speed $c_{uv}$, that is,
	\begin{eqnarray*}
	\left\{	\begin{array}{lll}
			\displaystyle\lim_{t\to +\infty}\sup_{|x|\le ct}
			\left\lbrace |u(t, x)-1| +
			|v(t, x)|\right\rbrace =0, &&\text{for}~
			0<c<c_{uv},\\
			\displaystyle\lim_{t\to +\infty}\sup_{c_{1}t\le|x|\le c_{2}t}
			\left\lbrace |u(t, x)| +|v(t, x)-1|\right\rbrace 
			=0, &&\text{for}~
			c_{uv}<c_{1}\le c_{2}<c_{v},\\
			\displaystyle\lim_{t\to +\infty}\sup_{|x|\ge ct}
			\left\lbrace |u(t, x)| +
			|v(t, x)|\right\rbrace =0, &&\text{for}~
			c>c_{v}.
		\end{array}
		\right.
	\end{eqnarray*}
For the critical case $c_u=c_v$, it is difficult to determine the behavior of the solution which is highly depending on both the parameters $d$, $r$ and the sizes of the initial supports. 

On the other hand, when $U$ is a general measurable set and $V$ is its complementary set, that is, $V=\R^N\setminus U$, we in \cite{BG} got a formula of spreading speed for $u$ and spreading sets for $u$ and $v$. If  \textbf{(A1)}-\textbf{(A2)} hold, $U$ satisfies $U_{\rho}\neq \emptyset$ and 
	$$\mathcal{B}(U)\cup	\mathcal{U}(U_{\rho})=
		\mathbb{S}^{N-1},$$
	where $\rho$ is given by Proposition~\ref{proposition1},  then for every $e\in \mathbb{S}^{N-1}$, the spreading speed of $u$ is given by
	\begin{equation}\label{unuve}
		w_{uv}(e)=\sup_{\xi\in \mathcal{U}
			(U), ~\xi\cdot e\ge 0}
		\frac{c_{uv}}{\sqrt{1-(\xi\cdot e)^{2}}}
	\end{equation}
with the conventions: $w_{uv}(e)=c_{uv}$ if there is no $\xi \in \mathcal{U}(U)$ such that $\xi \cdot e\ge 0$, and $w_{uv}(e)=+\infty$ if $e\in \mathcal{U}(U)$.
The spreading set of $u$ is then given by 
\be\label{W-uv}
\W_{uv}=\left\{ re:e\in \mathbb{S}^{N-1}, 0\le r<w_{uv}(e)\right\}=\mathbb{R}^{+}\mathcal{U}(U)+B_{c_{uv}}.
\ee
The map $e\mapsto w_{uv}(e)\in [c_{uv},+\infty)$ is continuous in $\mathbb{S}^{N-1}$ and $\partial \mathcal{W}_{uv}$ is also continuous.
Meantime, because of $U\cup V=\R^N$, the spreading set of $v$ is $\R^N\setminus \overline{\W_{uv}}$.

\subsection{Main results}

In this paper, we consider the more general case, say, $U$ and $V$ are all general measurable sets and do not necessarily satisfy $U\cup V=\R^N$. As we have mentioned, the spreading speeds of $u$ and $v$ for \eqref{pro1} are relying on the comparison of the speeds of single species. Recall that $w_u(e)$ and $w_v(e)$ given by \eqref{omega-u} and \eqref{omega-v} are the spreading speeds of single species $u$ and $v$ respectively in the direction $e\in\mathbb{S}^{N-1}$.  Based on the observation of the results in one dimension, it seems to be natural to expect that for every direction $e\in\mathbb{S}^{N-1}$, 
\begin{itemize}
\item[(i)] if $w_u(e)>w_v(e)$, then the spreading speed of $u$ is $w_u(e)$ and $v$ extincts in this direction;

\item[(ii)] if $w_u(e)<w_v(e)$, then $v$ spreads at speed $w_v(e)$ and $u$ invades it at speed $w_{uv}(e)$.
\end{itemize}
However, assertion (i) needs some additional conditions, because of the effect of unboundedness of initial supports and assertion (ii) is not true in general. 
Let us define the following notions.

\begin{definition}\label{def:Path}
For a nonempty set $U$ and a point $x$, let $\overline{x}$ be a projection of $x$ on $U$, that is, $|x-\overline{x}|=\hbox{\rm dist}(x,U)$. A projection path of $x$ onto $U$ is the set
$$P(x,U):=\{\lambda \overline{x}+(1-\lambda)x:\, \lambda\in [0,1]\}.$$ 
\end{definition}

If $x\in \overline{U}$, then $\overline{x}=x$ and $P(x,U)=\{x\}$. Clearly, the projection path of $x$ onto $U$ is not unique in general.

\begin{definition}\label{def:starshape}
The set $K$ is called star-shaped with respect to a closed set $U$ if  there is a projection path $P(x,U)\subset K$ for every $x\in K$.
\end{definition}

We call that the empty set $\emptyset$ is star-shaped with respect to any closed set. For nonempty sets $K$ and $U$, one should have $\partial K\cap \partial U\neq \emptyset$ if $K$ is star-shaped with respect to $\overline{U}$. By definitions of $\W_u$, $\W_v$ and $\W_{uv}$, one knows that they are all star-shaped with respect to $0$.

For $x\in\R^N\setminus \{0\}$, the notation $\widehat{x}$ means $x/|x|$. Then, we have the following theorem.

\begin{theorem}\label{Th1}
	Assume that {\rm \textbf{(A1)}-\textbf{(A2)}} 	hold, $U_{\rho}\neq\emptyset$ where $\rho$ is given by Proposition~\ref{proposition1}, $V_{\alpha}\neq \emptyset$ for some $\alpha>0$ and $U$, $V$ satisfy 
	\begin{equation}\label{eq:BUS}
			\mathcal{B}(U)\cup	\mathcal{U}(U_{\rho})=\mathbb{S}^{N-1},\quad 	\mathcal{B}(V)\cup \mathcal{U}(V_{\alpha})=\mathbb{S}^{N-1}.
	\end{equation}	
For every $e\in \mathbb{S}^{N-1}$, there holds that if there is a projection path $P(e,\R^+\mathcal{U}(U)\cup \{0\})$ such that 
\be\label{Pathwu>wv}
w_u(\widehat{\xi})>w_v(\widehat{\xi}), \hbox{ for all $\xi\in P(e,\R^+\mathcal{U}(U)\cup\{0\})\setminus\{0\}$},
\ee 
then
\begin{eqnarray}\label{Sspeed-u}
\left\{\begin{array}{lll}
\displaystyle\lim_{t\to +\infty}\sup_{0\le s\le c}\left \{ |u(t, ste)-1| + |v(t, ste)|\right \} =0, &\text{for }		0\le c<w_u(e),&\\
\displaystyle\lim_{t\to +\infty}\sup_{s\ge c} |u(t, ste) |=0, &\text{for }c>w_u(e),&\\
\displaystyle\lim_{t\to +\infty}\sup_{r\ge 0} |v(t, re)|  =0, &&
\end{array}
\right.
\end{eqnarray}
and if  there is a projection path $P(e,\R^+\mathcal{U}(V)\cup \{0\})$ such that 
\be\label{Pathwu<wv}
w_u(\widehat{\xi})<w_v(\widehat{\xi}), \hbox{ for all $\xi\in P(e,\R^+\mathcal{U}(V)\cup\{0\})\setminus \{0\}$},
\ee
then
\begin{eqnarray}\label{Sspeed-v}
\left\{\begin{array}{lll}
\displaystyle\lim_{t\to +\infty}\sup_{0\le s\le c}\left \{ |u(t, ste)-1| + |v(t, ste)|\right \} =0, &\text{for }		0\le c<w_{uv}(e),&\\
\displaystyle\lim_{t\to +\infty}\sup_{c_1\le s\le c_2}\left \{ |u(t, ste)| + |v(t, ste)-1|\right \} =0, &\text{for } w_{u}(e)<c_1<c_2<w_v(e),&\\
\displaystyle\lim_{t\to +\infty}\sup_{ s\ge c}\left \{ |u(t, ste)| + |v(t, ste)|\right \} =0, &\text{for } c>w_v(e).&
\end{array}
\right.
\end{eqnarray}
\end{theorem}

The condition \eqref{Pathwu>wv} means that not only the spreading speed of single species $u$ is larger than the spreading speed of single species $v$ in the direction $e$, but also in the directions $\widehat{\xi}$ with $\xi\in P(e,\R^+\mathcal{U}(U)\cup\{0\})\setminus\{e,0\}$. Then, under this condition, one can have that  the spreading speed of $u$ is $w_u(e)$ and $v$ extincts in the direction $e$. We explain roughly why we need this condition here. For a direction $e$ such that $w_u(e)>c_u$, we have mentioned that one can regard the species at $w_u(e)e$ as coming from $(w_u(e)\xi\cdot e)\xi$ with speed $c_u$ where $\xi\in \mathcal{U}(U)$ such that $\hbox{dist}(e,\R^+\mathcal{U}(U))=\sqrt{1-(\xi\cdot e)^2}$. So, on the path from $(w_u(e)\xi\cdot e)\xi$ to $w_u(e)e$, there should not be any species $v$, otherwise the competition slows down the speed of $u$. This is guaranteed by our condition. Notice that $(\xi\cdot e)\xi$ is a projection of $e$ on $\R^+\mathcal{U}(U)\cup \{0\}$. For a direction $e$ such that $w_u(e)=c_u$, the same explanation applies and $0$ is a projection of $e$ on $\R^+\mathcal{U}(U)\cup \{0\}$. If not only the spreading speed of single species $v$ is larger than the spreading speed of single species $v$ in the direction $e$, but also in the directions $\widehat{\xi}$ with $\xi\in P(e,\R^+\mathcal{U}(V)\cup\{0\})\setminus \{0\}$, that is, under condition~\eqref{Pathwu<wv}, one can have that $v$ spreads at speed $w_v(e)$ and $u$ invades it at a speed between $w_{uv}(e)$ and $w_u(e)$ in the direction $e$. The reason for the necessity of \eqref{Pathwu<wv} is similar as \eqref{Pathwu>wv}. Under \eqref{Pathwu<wv}, the invading speed of $u$ is less than $w_u(e)$ which is due to the competition and is possibly greater than $w_{uv}(e)$, because the spreading in directions $e'$ close to $e$ may accelerate the speed in the direction $e$ if $w_u(e')>w_v(e')$. However, the only information on the projection path $P(e,\R^+\mathcal{U}(V)\cup\{0\})$ is not enough to deduce the precise spreading speed of $u$.

To have a clear sight of the spreading for $u$ and $v$ and to have a precise spreading speed of $u$ for the directions $e$ such that $w_u(e)<w_v(e)$,  we deduce the spreading sets under stronger conditions. 

\begin{theorem}\label{Th2}
Assume that the assumptions in Theorem~\ref{Th1} hold. If 
\be\label{Starshape}
\begin{aligned}
\R^+(\W_u\setminus \overline{\W_v})&\cup \{0\} \hbox{ and } \R^+(\W_v\setminus \overline{\W_u})\cup \{0\} \hbox{ are star-shaped}\\
&\hbox{  with respect to $\R^+\mathcal{U}(U)\cup \{0\}$ and $\R^+\mathcal{U}(V)\cup \{0\}$ respectively},
\end{aligned}
\ee
and
\be\label{C-A}
\overline{\{e\in\mathbb{S}^{N-1}:\, w_u(e)>w_v(e)\}}=\{e\in\mathbb{S}^{N-1}:\, w_u(e)\ge w_v(e)\},
\ee 
then the spreading set of $u$ is given by 
\begin{equation}\label{Su}
\begin{aligned}
\mathcal{S}_u:=&\W_{uv}\cup \Big(\underset{\substack{e\in\mathbb{S}^{N-1}, \\ w_u(e)\ge w_v(e)}}{ \cup} \underset{0<\tau<1}{\cup} B_{\tau c_{uv}}((1-\tau)w_u(e)e)\Big)\\
=&\{re:\, e\in\mathbb{S}^{N-1},\, 0\le r<s_u(e)\},
\end{aligned}
\end{equation}
with the convention that $\cup_{0<\tau<1} B_{\tau c_{uv}}((1-\tau)w_u(e)e)=\R^+ e +B_{c_{uv}}$ if $w_u(e)=+\infty$,
and the spreading set of $v$ is given by 
\be\label{Sv}
\mathcal{S}_v:=\W_v\setminus \overline{\mathcal{S}_u}=\{re:\, e\in\mathbb{S}^{N-1}\hbox{ such that }w_u(e)<w_v(e),\, s_u(e)<r<w_v(e)\},
\ee
where
\be\label{F-su}
s_u(e):=\sup_{\substack{\xi\in\mathbb{S}^{N-1},\\ w_u(\xi)\ge w_v(\xi), \\ \xi\cdot e> c_{uv}/w_u(\xi)}}\sup_{0\le c<w_u(\xi)}\frac{c c_{uv}}{\sqrt{1-(\xi\cdot e)^2} \sqrt{c^2 +c_{uv}^2}+c_{uv}\xi\cdot e },
\ee
with the convention that $s_u(e)=+\infty$ if $e\in \mathcal{U}(U)$ and $s_u(e)=c_{uv}$ if there is no $\xi\in\mathbb{S}^{N-1}$ such that $w_u(\xi)>w_v(\xi)$ and $\xi\cdot e> c_{uv}/w_u(\xi)$. Moreover, the map $e\mapsto s_u(e)$ is continuous in $\mathbb{S}^{N-1}$, $s_u(e)=w_u(e)$ for very $e\in\mathbb{S}^{N-1}$ such that $w_u(e)\ge w_v(e)$ and $w_{uv}(e)\le s_u(e)<w_u(e)$ for every $e\in\mathbb{S}^{N-1}$ such that $w_u(e)<w_v(e)$.
\end{theorem}

\begin{figure}
\centering{
	\includegraphics[width=7 cm]{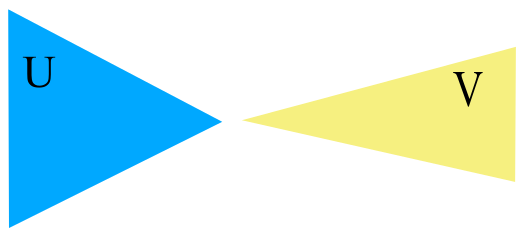}\quad \quad 
		\includegraphics[width=7 cm]{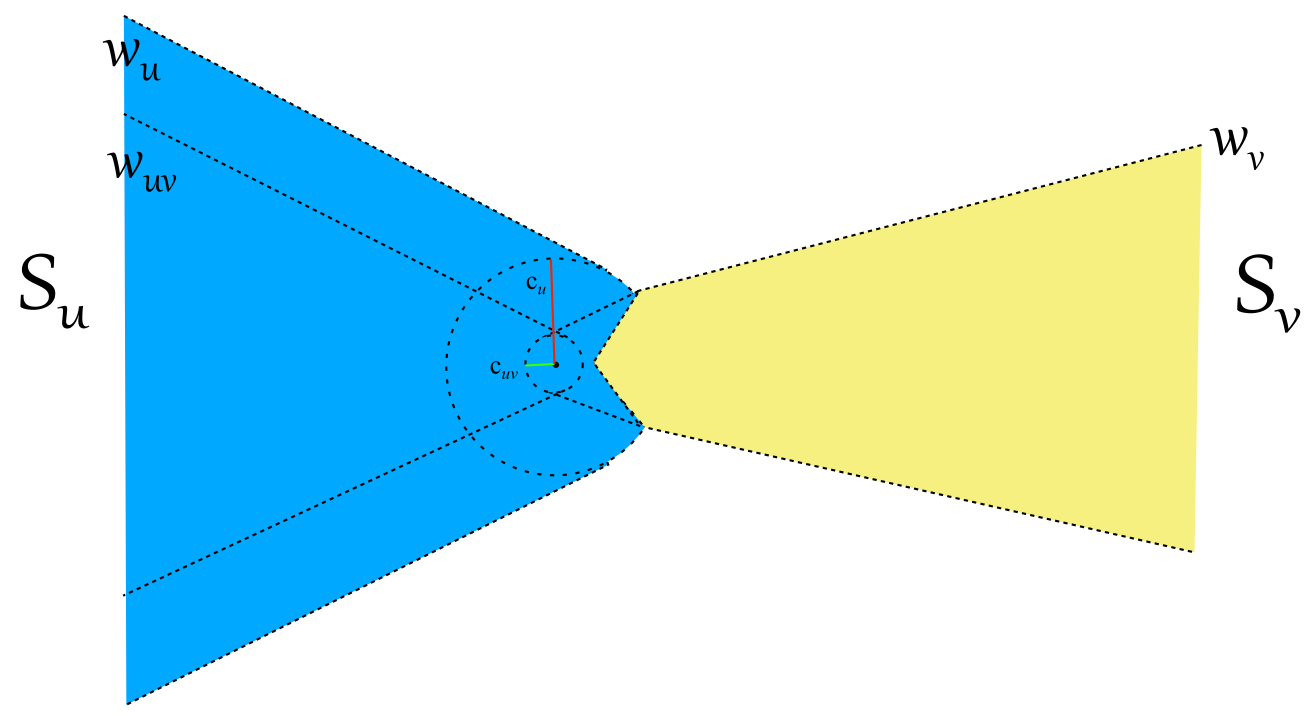}
		\caption[Fig. 1.]{Left: initial supports $U$ and $V$; Right: $\mathcal{S}_u$ (blue) and $\mathcal{S}_v$ (yellow).}\label{fig:1}}
\end{figure}

From the proof of Theorem~\ref{Th2}, we have 
$$\W_u=\underset{e\in\mathbb{S}^{N-1}}{\cup}\underset{0<\tau<1}{\cup}B_{\tau c_{uv}}((1-\tau) w_u(e) e),$$
 that is, \eqref{claim2-Wu}. By the definition of $\mathcal{S}_u$, it has the inclusion $\mathcal{S}_u\subset \W_u$. In particular, if $\{e\in\mathbb{S}^{N-1}:\, w_u(e)>w_v(e)\}=\mathbb{S}^{N-1}$, then $\mathcal{S}_u=\W_u$. Then, we have the following inclusion relation of sets $\W_{uv}\subset \mathcal{S}_u \subset \W_u$, $\mathcal{S}_v\subset \W_v$, see Figure~\ref{fig:1}. If $\{e\in\mathbb{S}^{N-1}:\, w_u(e)>w_v(e)\}=\emptyset$, then $\{e\in\mathbb{S}^{N-1}:\, w_u(e)\ge w_v(e)\}=\emptyset$ by \eqref{C-A} and
$$\underset{\substack{e\in\mathbb{S}^{N-1}, \\ w_u(e)\ge w_v(e)}}{ \cup} \underset{0<\tau<1}{\cup} B_{\tau c_{uv}}((1-\tau)w_u(e)e)=\emptyset,$$
which implies $\mathcal{S}_u=\W_{uv}$.
  If $\{e\in\mathbb{S}^{N-1}:\, w_u(e)>w_v(e)\}\neq \emptyset$ and $\mathcal{U}(U)=\emptyset$, then 
  $$\W_{uv}=B_{c_{uv}}\subset \underset{\substack{e\in\mathbb{S}^{N-1}, \\ w_u(e)\ge w_v(e)}}{ \cup} \underset{0<\tau<1}{\cup} B_{\tau c_{uv}}((1-\tau)w_u(e)e).$$
  If $\{e\in\mathbb{S}^{N-1}:\, w_u(e)>w_v(e)\}\neq \emptyset$ and $\mathcal{U}(U)\neq \emptyset$, then $w_u(e)=+\infty\ge w_v(e)$ for every $e\in\mathcal{U}(U)$ and
  $$\underset{\substack{e\in\mathbb{S}^{N-1}, \\ w_u(e)\ge w_v(e)}}{ \cup} \underset{0<\tau<1}{\cup} B_{\tau c_{uv}}((1-\tau)w_u(e)e)\supset \R^+\mathcal{U}(U)+B_{c_{uv}}=\W_{uv},$$
  by the convention of $\mathcal{S}_u$. So, in the case $\{e\in\mathbb{S}^{N-1}:\, w_u(e)>w_v(e)\}\neq \emptyset$, the set $\mathcal{S}_u$ can be written as
  $$\mathcal{S}_u=\underset{\substack{e\in\mathbb{S}^{N-1}, \\ w_u(e)\ge w_v(e)}}{ \cup} \underset{0<\tau<1}{\cup} B_{\tau c_{uv}}((1-\tau)w_u(e)e).$$
  Under conditions of Theorem~\ref{Th2}, we can immediately  have  the precise speed of species $u$ and $v$ in every direction.

\begin{corollary}\label{Cor1}
Assume that all conditions of Theorem~\ref{Th2} hold. For every $e\in \mathbb{S}^{N-1}$ such that $w_u(e)\ge w_v(e)$, it holds that
\begin{eqnarray}\label{su-full}
\left\{\begin{array}{lll}
\displaystyle\lim_{t\to +\infty}\sup_{0\le s\le c}\left \{ |u(t, ste)-1| + |v(t, ste)|\right \} =0, &\text{for }		0\le c<w_u(e),&\\
\displaystyle\lim_{t\to +\infty}\sup_{s\ge c} |u(t, ste) |=0, &\text{for }c>w_u(e),&\\
\displaystyle\lim_{t\to +\infty}\sup_{r\ge 0} |v(t, re)|  =0.
\end{array}
\right.
\end{eqnarray}
For every $e\in \mathbb{S}^{N-1}$ such that $w_u(e)< w_v(e)$, it holds that
\begin{eqnarray}\label{sv-full}
\left\{\begin{array}{lll}
\displaystyle\lim_{t\to +\infty}\sup_{0\le s\le c}\left \{ |u(t, ste)-1| + |v(t, ste)|\right \} =0, &\text{for }		0\le c<s_u(e),&\\
\displaystyle\lim_{t\to +\infty}\sup_{c_1\le s\le c_2}\left \{ |u(t, ste)| + |v(t, ste)-1|\right \} =0, &\text{for } s_{u}(e)<c_1<c_2<w_v(e),&\\
\displaystyle\lim_{t\to +\infty}\sup_{ s\ge c}\left \{ |u(t, ste)| + |v(t, ste)|\right \} =0, &\text{for } c>w_v(e).&
\end{array}
\right.
\end{eqnarray}
\end{corollary}

\begin{remark}
In general, for directions $e\in\mathbb{S}^{N-1}$ such that $w_u(e)=w_v(e)$, it is difficult to determine the behavior of the solution as the critical case $c_u=c_v$ in one dimensional space. However, under conditions of Theorem~\ref{Th2}, even for directions $e\in\mathbb{S}^{N-1}$ such that $w_u(e)=w_v(e)$, one can still have that the spreading speed of $u$ is $w_u(e)$, see \eqref{su-full}. The condition \eqref{C-A} especially means that $e$ can be approximated by a sequence of directions $\{e_n\}_{n\in\mathbb{N}}$ such that $w_u(e_n)>w_v(e_n)$.
\end{remark}

\subsection{Comments on Theorems~\ref{Th1}-\ref{Th2}}\label{subs-1.4}

Formulae of the spreading speeds and sets \eqref{omega-u}, \eqref{omega-v}, \eqref{unuve}, \eqref{W-u}, \eqref{W-v} and \eqref{W-uv} are valid under the condition \eqref{eq:BUS}. We refer to \cite[Proposition~5.1]{HR} for some sufficient conditions such that  \eqref{eq:BUS} holds. However, there may exist a direction such that it is neither bounded or unbounded for a set in general. For instance, for any given $e\in\mathbb{S}^{N-1}$, if $U=B_1+\{2^n e:\, n\in\mathbb{N}\}$, then $e\not\in \mathcal{B}(U)\cup \mathcal{U}(U)$ and formulae \eqref{omega-u}, \eqref{W-u} fail. There are also examples violating \eqref{eq:BUS} but formulae \eqref{omega-u}, \eqref{W-u} are still valid, see \cite{HR}. 

For conditions \eqref{Pathwu>wv} and \eqref{Pathwu<wv}, we have roughly explained their necessity. Unfortunately, we can not rigorously prove a counterexample. We leave it as a future research subject.

By Definitions~\ref{def:Path}-\ref{def:starshape}, condition \eqref{Starshape} implies that \eqref{Pathwu>wv} holds for every $e\in\mathbb{S}^{N-1}$ such that $w_u(e)>w_v(e)$ and \eqref{Pathwu<wv} holds for every $e\in\mathbb{S}^{N-1}$ such that $w_u(e)<w_v(e)$. So, condition \eqref{Starshape} is stronger than conditions \eqref{Pathwu>wv}-\eqref{Pathwu<wv}.

The condition~\eqref{C-A} means that the $(N-1)$-dimensional Hausdorff measure of $\{e\in\mathbb{S}^{N-1}:\, w_u(e)= w_v(e)\}$ is zero. A counterexample is $\W_u=B_{c_u}=\W_v=B_{c_v}$ when $c_u=c_v$. If \eqref{C-A} is not satisfied,  it is difficult to determine the behavior of the solution for directions $e\in\mathbb{S}^{N-1}$ such that $w_u(e)=w_v(e)$.

When $U$ and $V$ are compact sets with $U_{\rho}\neq \emptyset$ and $V_{\alpha}\neq \emptyset$, then $\mathcal{B}(U)=\mathcal{B}(V)=\mathbb{S}^{N-1}$ and $\mathcal{U}(U_{\rho})=\mathcal{U}(V_{\alpha})=\emptyset$ which satisfies \eqref{eq:BUS}. Moreover, $\W_u=B_{c_u}$, $\W_v=B_{c_v}$ and $\W_{uv}=B_{c_{uv}}$. If $c_u>c_v$, \eqref{Starshape} and \eqref{C-A} are satisfied by noticing $\R^+\mathcal{U}(U)\cup \{0\}=\R^+\mathcal{U}(V)\cup \{0\}=\{0\}$. Then, by Theorem~\ref{Th2}, $s_u(e)=c_u$ for every $e\in\mathbb{S}^{N-1}$ and $\mathcal{S}_u=B_{c_u}$, $\mathcal{S}_v=\emptyset$. If $c_u<c_v$, \eqref{Starshape} and \eqref{C-A} are also satisfied. By Theorem~\ref{Th2}, $s_u(e)=c_{uv}$ for every $e\in\mathbb{S}^{N-1}$ and $\mathcal{S}_u=B_{c_{uv}}$, $\mathcal{S}_v=B_{c_v}\setminus \overline{B_{c_{uv}}}$. As we have mentioned, it is hard to determine the behavior of the solution for the case $c_u=c_v$. However, if $c_u=c_v$ but $U$ and $V$ are not compact, it is possible to have the spreading speeds and sets. For instance, if $c_u=c_v$,
$$U=\{x\in\R^N:\, x_N<0,\, |x'|<R\} \hbox{ and } V=\{x\in\R^N:\, x_N>0,\, |x'|<R\}$$
with $R>0$ sufficiently large, then $\W_u=\R^+ (0,0,\cdots,-1)+B_{c_u}$ and $\W_v=\R^+ (0,0,\cdots,1)+B_{c_v}$. Clearly, \eqref{Starshape} and \eqref{C-A} are satisfied. Then, 
\be 
\begin{aligned}
\mathcal{S}_u=&\Big\{x\in\R^N:\, 0\le x_N\le \sqrt{c_{uv}^2-|x'|^2} \hbox{ for } |x'|\le \frac{c_{uv}^2}{c_u},\, 0\le x_N\le \frac{c_{uv}}{\sqrt{c_u^2-c_{uv}^2}}(c_u-|x'|) \\
&\hbox{ for } \frac{c_{uv}^2}{c_u}<|x'|<c_u\Big\}\cup \{x\in\R^N:\, x_N<0,\, |x'|<c_u\},
\end{aligned}
\ee
and $\mathcal{S}_v:=\W_v\setminus \overline{\mathcal{S}_u}$. When $U$ is a general measurable set and $V=\R^N\setminus U$ satisfying \eqref{eq:BUS}, the $(N-1)$-dimensional Hausdorff measure of $\overline{\mathcal{B}(U)}\cap \mathcal{U}(U)$ is equal to zero since $\mathcal{B}(U)$ and $\mathcal{U}(U)$ are relatively open and closed in $\mathbb{S}^{N-1}$ respectively. Then, $\hbox{int}(\mathcal{U}(U))=\mathcal{B}(V)$ ($\hbox{int}(\mathcal{U}(U))$ denotes interiors of $\mathcal{U}(U)$ on $\mathbb{S}^{N-1}$) and $\overline{\mathcal{B}(U)}=\mathcal{U}(V)$. Moreover, $\W_u=\R^+\mathcal{U}(U)+B_{c_u}$, $\W_v=\R^+\mathcal{U}(V)+B_{c_v}=\R^+ \mathcal{B}(U)+B_{c_v}$ and $\W_{uv}=\R^+\mathcal{U}(U)+B_{c_{uv}}$. In other words,  $w_u(e)=+\infty>w_v(e)$ for $e\in \hbox{int}(\mathcal{U}(U))$, $w_u(e)<+\infty=w_v(e)$ for $e\in \mathcal{B}(U)$ and $w_u(e)=+\infty=w_v(e)$ for $e\in  \overline{\mathcal{B}(U)}\cap \mathcal{U}(U)$. Conditions \eqref{Starshape} and \eqref{C-A}  are then satisfied. From Theorem~\ref{Th2}, the spreading set $\mathcal{S}_u=\W_{uv}$ and $\mathcal{S}_v=\R^N\setminus \overline{\W_{uv}}$ which is consistent with the results of \cite{BG}.

\vskip 0.3cm
\noindent \textbf{Outline of the paper.} In Section~2, we show some preliminary tools. Section~3 is devoted to the proofs of Theorems~\ref{Th1} and \ref{Th2}.

\section{Preliminaries}
In this section, we deduce some preliminary tools which will be used frequently in the proofs of Theorems~\ref{Th1}-\ref{Th2}.

We first show an estimate of the solution when points stay an uniform distance away from a positive distance interior of the initial support.

\begin{lemma}\label{delta-uv}
Let $(u,v)(t,x)$ be the solution of \eqref{pro1}. For any $t\in\R$, $R>0$ and $\alpha>0$ such that $U_{\alpha},\, V_{\alpha}\neq \emptyset$, there is $\delta\in (0,1)$ such that
$$u(t,x)\ge \delta, \hbox{ for any $x\in \R^N$ such that $\hbox{dist}(x,U_{\alpha})< R$},$$
and
$$v(t,x)\ge \delta, \hbox{ for any $x\in \R^N$ such that $\hbox{dist}(x,V_{\alpha})< R$}.$$
\end{lemma}

\begin{proof}
We only prove the lemma for $u$ and the same arguments can be applied for $v$. 

Since $0\le v(t,x)\le 1$, the solution $\widetilde{u}(t,x)$ of the heat equation
$$\widetilde{u}_t=d\Delta \widetilde{u}, \quad  t>0,\, x\in\R^N, \hbox{ with $\widetilde{u}(0,x)=\1_U$},$$
is a subsolution for $u$. Thus, we have 
$$u(t,x)\ge \widetilde{u}(t,x)=\int_{\R^N} \Gamma(t,x-y) \widetilde{u}(0,y)dy=\int_{U} \Gamma(t,x-y)dy,$$
where $\Gamma(s,\xi)$ is the heat kernel. For any $x\in \R^N$ such that $\hbox{dist}(x,U_{\alpha})< R$, there is $x_0\in U_{\alpha}$ such that $|x-x_0|\le R$ and $B_{\alpha}(x_0)\subset U$. It follows that for any $x\in \R^N$ such that $\hbox{dist}(x,U_{\alpha})< R$,
$$u(t,x)\ge \int_{B_\alpha(x_0)} \Gamma(t,x-y)dy\ge \frac{1}{\sqrt{4\pi t}} \int_{B_\alpha(x_0)} e^{-\frac{|x-y|^2}{4t}} dy\ge \frac{1}{\sqrt{4\pi t}} e^{-\frac{(R+\alpha)^2}{4t}}\int_{B_\alpha(x_0)}  dy.$$
The conclusion then follows.
\end{proof}

Then, we need two lemmas of \cite{BG} which are  about two types of initial value problems.
For $\rho>0$ and $\delta\in (0,1)$, let $(\underline{u}_{\rho},\overline{v}_{\rho})(t,x)$ be the solution of (\ref{pro1})  with the following initial condition
\begin{eqnarray}\label{eq-uv-rho}
	\left\{\begin{array}{lll}
		(\underline{u}_{\rho}, \overline{v}_{\rho})(0,x)
		=(1-\delta, \delta), && \hbox{for } x\in B_{\rho},\\
		(\underline{u}_{\rho}, \overline{v}_{\rho})(0,x)		=(0, 1), &&  \hbox{for } x\in \mathbb{R}^{N}\setminus B_{\rho}.
	\end{array}
	\right.
\end{eqnarray}
By \cite[Lemma~2.1]{BG}, one has that the region where $(\underline{u}_{\rho}, \overline{v}_{\rho})(t,x)$ is close to $(1,0)$ is expanding at an approximate speed $c_{uv}$.

\begin{lemma}\label{lemma:sub}
	There are $\rho>0$ and $\delta_0\in (0,1)$ such that for any $\varepsilon\in (0,c_{uv})$ and $\delta\in (0,\delta_0)$, there holds that
	\begin{equation}\label{extending}
		(\underline{u}_{\rho},\overline{v}_{\rho})(t,x)\rightarrow (1,0)~\text{
			uniformly in}~\left\lbrace x\in \mathbb{R}^{N}; |x|\le 
		(c_{uv}-\varepsilon)t\right\rbrace ~\text{as}~t\to +\infty.
	\end{equation}
	Moreover,
	\begin{equation}\label{lemma2.1-delta}
		\begin{aligned}
			\underline{u}_{\rho}(t,x)\ge 1-2\delta,~\overline{v}_{\rho}(t,x)(t, x)\le 2\delta, \quad &\text{for all} ~t\ge 0\hbox{ and }  |x|\le (c_{uv}-\varepsilon)t.
		\end{aligned}
	\end{equation}
\end{lemma}

\begin{remark}
Proposition~\ref{proposition1} is an immediate consequence of Lemma~\ref{lemma:sub} and the comparison principle.
\end{remark}

For $R>0$ and $\delta\in (0,1)$, let $(\overline{u}_{R},\underline{v}_{R})(t,x)$ be the solution of (\ref{pro1})  with the following initial condition
\begin{eqnarray*}
	\left\{\begin{array}{lll}
		(\overline{u}_{R}, \underline{v}_{R})(0,x)
		=(\delta,1-\delta), && \hbox{for } x\in B_{R},\\
		(\overline{u}_{R}, \underline{v}_{R})(0,x)	=(1,0), &&  \hbox{for } x\in \mathbb{R}^{N}\setminus B_{R},
	\end{array}
	\right.
\end{eqnarray*}
By \cite[Lemma~2.2]{BG}, one has that the region where $(\overline{u}_{R}, \underline{v}_{R})(t,x)$ is close to $(0,1)$ is contracting at an approximate speed $c_{uv}$. 

\begin{lemma}\label{lem1}
	For any $\varepsilon>0$, there exist  $R_{\varepsilon}>0$ and $\delta_0\in (0,1)$ such that for any $R\ge R_{\varepsilon}$ and $\delta\in (0,\delta_0)$, 
	there holds 
	\begin{equation}\label{delta1}
		\begin{aligned}
			\overline{u}_{R}(t,x)\le 2\delta,~\underline{v}_{R}(t,x)(t, x)\ge 1-2\delta, \quad &\text{for} ~0\le t\le 
			\frac{R-R_{\varepsilon}}{c_{uv}+\varepsilon}\\
			&\text{and} ~|x|\le R-R_{\varepsilon}-(c_{uv}+\varepsilon)t.
		\end{aligned}
	\end{equation}
	Moreover, there exists $T_{\varepsilon}>0$ such that 
	\begin{equation}
		\begin{aligned}
			\overline{u}_{R}(t,x)\le \delta,~\underline{v}_{R}(t, x)\ge 1-\delta, \quad &\text{for} ~T_{\varepsilon}\le t\le 
			\frac{R-R_{\varepsilon}}{c_{uv}+\varepsilon}\\
			&\text{and} ~|x|\le R-R_{\varepsilon}-(c_{uv}+\varepsilon)t.
		\end{aligned}
	\end{equation}
\end{lemma}

If there is no competition, say, $v\equiv 0$ for instance, then species $u$ spreads approximately as $tB_{c_u}$. Now, we show that for any $C\subset B_{c_u}\setminus\{0\}$,  if the density of $v$ stays low in $Ct$, then species $u$ spreads at $Ct$.  
Let $\Phi_R: \mathbb{R}^{N}\rightarrow\mathbb{R}$ be the principal eigenfunction of the Laplace
 operator on the ball $B_{R}$ with Dirichlet boundary conditions, normalized by $||\Phi_{R}||_\infty=1$, satisfying 
 \begin{equation}
 	\begin{cases}
 		\Delta\Phi_R=\lambda_R\Phi_R &\hbox{in} ~B_R,\\
 		\Phi_R>0 &\hbox{in} ~B_R,\\
 	    \Phi_R=0 &\hbox{on} ~\partial B_R.
 	\end{cases}
 \end{equation} 
It is well known that $\lambda_R<0$ and it tends to 0 as $R\to +\infty$. We extend $\Phi_R$ on the whole space by setting $\Phi_R(x) = 0$ for  $|x| > R$.
For $\sigma\in (0,1)$, $e\in\mathbb{S}^{N-1}$, $0<c<c_u$ and $(t,x)\in [0,+\infty)\times\R^N$, define
\be\label{uRe}
\underline{u}_{c,e}(t, x)=\sigma e^{-\frac{c}{2d}(x\cdot e-ct) }\Phi_{R}\left( x -c t e \right).
\ee

\begin{lemma}\label{Rsigma}
For any compact set $C\subset B_{c_u}\setminus\{0\}$, there is $R>0$ and $\sigma_0>0$ such that for any $c>0$ and $e\in\mathbb{S}^{N-1}$ such that $ce\in C$ and any $\sigma\in (0,\sigma_0]$, if 
\be\label{con-v}
v(t,x)\le \frac{1}{4a}\Big(1-\frac{\max_{\xi\in C}\xi^2}{c_u^2}\Big), \hbox{ for $t \ge 0$ and $x\in C t + B_R$},
\ee
 then the function $\underline{u}_{c,e}(t,x)$ satisfies
\be\label{ineq-subs}
(\underline{u}_{c,e})_t-d\Delta\underline{u}_{c,e}-r\underline{u}_{c,e}(1-\underline{u}_{c,e}-av)\le 0.
\ee
for $(t,x)\in [0,+\infty)\times\R^N$.
\end{lemma}

\begin{proof}
Simply calculations yield that 
\begin{eqnarray*}
\begin{array}{lll}
 N[\underline{u}_{c,e},v]&:=&(\underline{u}_{c,e})_t-d\Delta\underline{u}_{c,e}-r\underline{u}_{c,e}(1-\underline{u}_{c,e}-av)\\
&=&  \frac{c^2}{2d} \underline{u}_{c,e} -  c\sigma e^{-\frac{c}{2d}(x\cdot e-ct)} \nabla\Phi_R\cdot e -d \Big(\frac{c^2}{4d^2} \underline{u}_{c,e} + \sigma e^{-\frac{c}{2d}(x\cdot e-ct)} \Delta\Phi_R\\
&& -c\sigma e^{-\frac{c}{2d}(x\cdot e-ct)} \nabla\Phi_R\cdot e \Big) -r \underline{u}_{c,e} (1-\underline{u}_{c,e} -a v)\\
&=&\Big(\frac{c^2}{4d} -d\lambda_R -r +r\underline{u}_{c,e} +a r v\Big)\underline{u}_{c,e} .
\end{array}
\end{eqnarray*}
For $t \ge 0$ and $x\in \R^N\setminus B_R(ct e)$, one has that $\underline{u}_{c,e}(t,x)\equiv 0$ and hence, $N[\underline{u}_{c,e},v]=0$. On the other hand, since $C$ is a compact subset of $B_{c_u}\setminus\{0\}$,  then 
$$\frac{c^2}{4d}\le \frac{\max_{\xi\in C}\xi^2}{4d}<\frac{c_u^2}{4d}=r.$$
Since $\lambda_R\rightarrow 0$ as $R\rightarrow +\infty$, $\sup_{t>0,x\in\R^N}\underline{u}_{c,e}(t,x)\rightarrow 0$ as $\sigma \rightarrow 0$ for fixed $R$ and any $e\in\mathbb{S}^{N-1}$, one can take $R>0$ sufficiently large and $\sigma_0>0$ sufficiently small such that 
$$\max\Big(-d\lambda_R, r \sup_{t>0,x\in\R^N}\underline{u}_{c,e}\Big)\le \frac{1}{4}\Big(r-\frac{\max_{\xi\in C}\xi^2}{4d}\Big),$$
for all $0<\sigma\le \sigma_0$.
Notice that $B_R(ct e) \subset Ct +B_R$.
For $t \ge 0$ and $x\in B_R(ct e)$, it follows from \eqref{con-v} and above inequality that $N[\underline{u}_{c,e},v]\le 0$.
\end{proof}

The same conclusion holds for $v$. For $\sigma\in (0,1)$, $e\in\mathbb{S}^{N-1}$, $0<c<c_v$ and $(t,x)\in [0,+\infty)\times\R^N$, define
\be
\underline{v}_{c,e}(t, x)=\sigma e^{-\frac{c}{2}(x\cdot e-ct) }\Phi_{R}\left( x -c t e \right).
\ee

\begin{lemma}\label{Lemma:Rsigma-v}
For any compact set $C\subset B_{c_v}\setminus\{0\}$, there is $R>0$ and $\sigma_0>0$ such that for any $c>0$ and $e\in\mathbb{S}^{N-1}$ such that $ce\in C$ and any $\sigma\in (0,\sigma_0]$, if 
\be
u(t,x)\le \frac{1}{4b}\Big(1-\frac{\max_{\xi\in C}\xi^2}{c_v^2}\Big), \hbox{ for $t \ge 0$ and $x\in C t + B_R$},
\ee
 then the function $\underline{v}_{c,e}(t,x)$ satisfies
\be
(\underline{v}_{c,e})_t-\Delta\underline{v}_{c,e}-\underline{v}_{c,e}(1-\underline{v}_{c,e}-b u)\le 0.
\ee
for $(t,x)\in [0,+\infty)\times\R^N$.
\end{lemma}

\section{Proofs of Theorems~\ref{Th1}-\ref{Th2}}
This section is devoted to the proofs of Theorems~\ref{Th1}-\ref{Th2}. Since conditions of Theorem~\ref{Th2} are stronger than Theorem~\ref{Th1}, all lemmas under conditions of Theorem~\ref{Th1} work for Theorem~\ref{Th2}. Throughout this section, $(u,v)(t,x)$ always denotes the solution of \eqref{pro1}.

\subsection{Proof of Theorem~\ref{Th1}}\label{Section1.7}
We first deal with Theorem~\ref{Th1}. From the results of single species \cite{HR},  we can easily have that $\W_u$ and $\W_v$ are spreading supersets of $u$ and $v$ respectively.

\begin{lemma}\label{uv0}
Assume that assumptions of Theorem~\ref{Th1} hold. For $e\in\mathcal{B}(U)$ and any $\varepsilon>0$, it holds that
 $$\sup_{x\in\mathcal{C}^{\varepsilon}_u(e)} u(t,tx)\rightarrow 0 \quad \hbox{ as $t\rightarrow +\infty$},$$
 where $\mathcal{C}^{\varepsilon}_u(e):=\cup_{\tau>1} B_{c_u(\tau-1)}(\tau (w_u(e)+\varepsilon) e)$. For $e\in\mathcal{B}(V)$ and any $\varepsilon>0$, it holds that
 $$\sup_{x\in\mathcal{C}^{\varepsilon}_v(e)} v(t,tx)\rightarrow 0 \quad \hbox{ as $t\rightarrow +\infty$},$$
 where $\mathcal{C}^{\varepsilon}_v(e):=\cup_{\tau>1} B_{c_v(\tau-1)}(\tau (w_v(e)+\varepsilon) e)$.
Moreover, sets $\W_u$ and $\W_v$ defined by \eqref{W-u} and \eqref{W-v} are spreading supersets of $u$ and $v$ respectively.
\end{lemma}

\begin{proof}	
The maximum principle for the system \eqref{pro1} implies that $0\le u(t,x)\le 1$ and $0\le v(t,x)\le 1$.	Let $\tilde{u}(t, x),$ be the solution of 
$$ \tilde{u}_t=d\Delta \tilde{u}+r\tilde{u}(1-\tilde{u}),\quad t>0,\, x\in\R^N,$$
with the initial condition $\tilde{u}(0,x) =\mathbbm{1}_U$. Because of the hair trigger effect and since $U\supset U_{\rho}\neq \emptyset$ with $\rho>0$ given by Proposition~\ref{proposition1}, it follows from \cite[Lemma~4.4]{HR} that 
$$\sup_{x\in\mathcal{C}^{\varepsilon}_u(e)} \widetilde{u}(t,tx)\rightarrow 0 \quad \hbox{ as $t\rightarrow +\infty$},$$
and from \cite[Theorem~2.2]{HR} that $\W_u$ defined by \eqref{W-u} is the spreading set of $\tilde{u}$. Especially,
$$
\lim_{t\to+\infty}\,\Big(\max_{x\in C} \tilde{u}(t,tx)\Big)=0,  \, \text{ for any non-empty compact set }C\subset\R^N\setminus \overline{\mc{W}_u}.
$$

Clearly, $(\tilde{u},0)(t,x)$ is a supersolution for the system \eqref{pro1}. Hence, $u(t,x)\le \tilde{u}(t,x)$ for $t\ge 0$ and $x\in\R^N$ and the conclusions follow. 

The same arguments deduce the same conclusions for $v$. 
\end{proof} 

We make the conclusions of Lemma~\ref{uv0} stronger such that one direction $e$ could be all directions of a neighbor of $e$.

\begin{lemma}\label{uv1}
Assume that assumptions of Theorem~\ref{Th1} hold. For any closed subset $\mathcal{B}_1$ of $\mathcal{B}(U)$ and $\varepsilon>0$, it holds that
 $$\sup_{x\in \underset{e\in \mathcal{B}_1}{\cup}\mathcal{C}^{\varepsilon}_u(e)} u(t,tx)\rightarrow 0 \quad \hbox{ as $t\rightarrow +\infty$}$$
  where $\mathcal{C}^{\varepsilon}_u(e)$ is defined in Lemma~\ref{uv0}.
 For any closed subset $\mathcal{B}_2$ of $\mathcal{B}(V)$ and $\varepsilon>0$, it holds that
 $$\sup_{x\in \underset{e\in \mathcal{B}_2}{\cup}\mathcal{C}^{\varepsilon}_v(e)} v(t,tx)\rightarrow 0 \quad \hbox{ as $t\rightarrow +\infty$},$$
  where $\mathcal{C}^{\varepsilon}_v(e)$ is defined in Lemma~\ref{uv0}.
\end{lemma}

\begin{proof}
Since $\mathcal{B}(U)$ is open and $\mathcal{B}_1$ is closed, there is an open set $\mathcal{B}'$ such that $\mathcal{B}_1\subset \mathcal{B}'\subset \mathcal{B}(U)$. By the definition of $\mathcal{C}^{\varepsilon}_u(e)$ and the continuity of $w_u(e)$, one has that $\overline{\cup_{e\in\mathcal{B}_1} \mathcal{C}^{\varepsilon}_u(e)}\subset \cup_{e\in\mathcal{B}'} \mathcal{C}^{\varepsilon}_u(e)$. Then, for $R>\sup_{e\in\mathcal{B}'} w_u(e) +\varepsilon$, because the set $\overline{\cup_{e\in\mathcal{B}_1} \mathcal{C}^{\varepsilon}_u(e)}\cap \partial B_R$ is compact, there is a finite number of $\mathcal{C}^{\varepsilon}_u(e_i)\cap \partial B_R$ with some $e_i\in \mathcal{B}'$ ($i=1,\cdots, n$) such that $\overline{\cup_{e\in\mathcal{B}_1} \mathcal{C}^{\varepsilon}_u(e)}\cap \partial B_R$ is covered by $\cup_{i=1}^n \mathcal{C}^{\varepsilon}_u(e_i)\cap \partial B_R$. Since $\mathcal{C}^{\varepsilon}_u(e)$ is a cone, it implies that $\cup_{e\in \mathcal{B}_1}\mathcal{C}^{\varepsilon}_u(e)\setminus B_R \subset \cup_{i=1}^n \mathcal{C}^{\varepsilon}_u(e_i)$. While, it is clear that $(\cup_{e\in\mathcal{B}_1} \mathcal{C}^{\varepsilon}_u(e))\cap B_R$ is a subset of $\R^N\setminus \overline{\mathcal{W}_u}$ with a positive distance against $\mathcal{W}_u$. So, $(\cup_{e\in\mathcal{B}_1} \mathcal{C}^{\varepsilon}_u(e))\cap B_R$ can be covered by a compact set $C\subset \R^N\setminus\overline{\mathcal{W}_u}$.  Then, by Lemma~\ref{uv0}, we can immediately have the conclusion. 
\end{proof}

For the spreading of $u$, the worst case is that there is competition everywhere (that is, $V$ is the complimentary set of $U$), but even in this case, $u$  has a spreading set $\W_{uv}$. So,  $\W_{uv}$ is always a spreading subset of $u$ and $\R^N\setminus \overline{\W_{uv}}$ is always a spreading superset of $v$.

\begin{lemma}\label{sycuv}
Under assumptions of Theorem~\ref{Th1},	the set $\mathcal{W}_{uv}$ defined by \eqref{W-uv} is a spreading subset of $u$ and the set $\W_v\setminus \overline{\W_{uv}}$ is a spreading superset of $v$.
\end{lemma}
\begin{proof}
Take any point $x_0\in U_\rho$. So, $B_{\rho}(x_0)\subset U$. By the comparison principle, we have $u(t, x)\ge \underline u_\rho(t, x-x_0), v(t, x)\le \overline v_\rho(t, x-x_0)$ for all $x\in \mathbb{R}^{N}$, where $(\underline{u}_{\rho},\overline{v}_{\rho})$ is the solution of \eqref{eq-uv-rho}. By Lemma \ref{lemma:sub}, it then follows that 
\be\label{Bcuv}
	\lim_{t\to +\infty}\sup_{x\in U_\rho+B_{ct}}\left(|u(t, x)-1|+|v(t, x)|\right)=0, \quad \hbox{for any}\quad
	0<c<c_{uv}.
\ee
	
Clearly, one has 
$$\mathcal{U}(U)\supset \mathcal{U}(U_{\rho})=\mathbb{S}^{N-1}\setminus \mathcal{B}(U)\supset \mathcal{U}(U),$$	
that is, $\mathcal{U}(U)=\mathcal{U}(U_{\rho})$. Take any $0<c'<c<c_{uv}$ and $\xi\in\mathcal{U}(U)=\mathcal{U}(U_{\rho})$. By the definition of $\mathcal{U}(U_{\rho})$, one has
$$\frac{1}{\lambda t}\hbox{dist}( \lambda\xi t,U_{\rho})\rightarrow 0, \quad \hbox{as $t\rightarrow +\infty$ for any fixed $\lambda>0$}.$$
Thus, $B_{c' t}(\lambda \xi t)\subset U_\rho+B_{ct}$ for large $t$. Notice that this inclusion also holds for $\lambda=0$. By \eqref{Bcuv}, it follows that
$$\lim_{t\to +\infty}\sup_{x\in B_{c' t}(\lambda \xi t)}\left(|u(t, x)-1|+|v(t, x)|\right)=0, \quad \hbox{for any $\lambda\ge 0$}.$$

 On the other hand, for any compact set $\mathcal{C}\subset \W_{uv}=\mathbb{R}^{+}\mathcal{U}(U)+B_{c_{uv}}$, it can be covered by finitely many balls of the form $B_{c'}(\lambda \xi)$, that is, $\mathcal{C}\subset \cup_{i=1}^m B_{c_i}(\lambda_i \xi_i)$. Then, 
$$\lim_{t\to+\infty}\,\Big(\max_{x\in C} \left(|u(t, xt)-1|+|v(t, xt)|\right)\Big)=0, $$ 
Therefore, $\mathcal{W}_{uv}$ is a spreading subset of $u$ and $\R^N\setminus \overline{\W_{uv}}$ is a spreading superset of $v$. Since $\W_v$ is a spreading superset of $v$, one has that $\W_v\cap (\R^N\setminus \overline{\W_{uv}})=\W_v\setminus \overline{\W_{uv}}$ is a spreading superset of $v$.
\end{proof}

Now, for directions $e\in\mathbb{S}^{N-1}$ such that $w_u(e)>w_v(e)$, we prove some spreading properties of the solution $(u,v)(t,x)$ under the condition \eqref{Pathwu>wv}. If the projection of $e$ on $\R^+\mathcal{U}(U)\cup\{0\}$ is $0$, $P(e,\R^+\mathcal{U}(U)\cup\{0\})=\{\lambda e:\, 0\le \lambda\le 1\}$ and \eqref{Pathwu>wv} is automatically satisfied since $\hat{\xi}=e$ for every $\xi\in P(e,\R^+\mathcal{U}(U)\cup \{0\})\setminus\{0\}$ and $w_u(e)=c_u$ from \eqref{omega-u}. Then we have the following conclusion. 

\begin{lemma}\label{lemma:bar-e=0}
Assume that assumptions of Theorem~\ref{Th1} hold. For every $e\in \mathbb{S}^{N-1}$ such that $w_u(e)=c_u>w_v(e)$ and $w_v(e)<c<c_u$,  there is $\varepsilon>0$ such that
$$\lim_{t\to +\infty}\sup_{x\in B_{\varepsilon}(ce)}\left \{ |u(t, xt)-1| + |v(t, xt)|\right \} =0.$$
\end{lemma}

\begin{proof}
Since $w_u(e)=c_u$ and $w_v(e)<c_u$, one has $e\in \mathcal{B}(U)$ and $e\in \mathcal{B}(V)$. Take 
$$w_v(e)<c<c'<c_u.$$
and
\be\label{epsilon}
0<\varepsilon<\min\left(1,\frac{w_v(e)}{2},\frac{c-w_v(e)}{3},\frac{c'-c}{2},\frac{c_u-c'}{2},c\Big(\frac{c_u}{c'}-1\Big)\right).
\ee
Since $\partial\W_v$ is continuous, even if it means decreasing $\varepsilon$, assume that $\overline{B_{\varepsilon}(ce)}\subset \R^N\setminus \overline{\W_v}$. By Lemma~\ref{uv0}, one has
\be\label{Bev0}
\lim_{t\rightarrow +\infty}\sup_{x\in B_{\varepsilon}(ce)} |v(t,xt)|=0.
\ee
 By \eqref{epsilon},  $c>w_v(e)+3\varepsilon$.
Take $M>1$ sufficiently large such that $\hat{c}(1-\frac{1}{M})>w_v(e)+2\varepsilon$ for any $\hat{c}\ge w_v(e)+3\varepsilon$. 
For $s\in [T, +\infty)$ with $T>0$ to be chosen, consider the set
\be\label{Qs}
Q_s:= \underset{\hat{c}\ge w_v(e)+3\varepsilon}{ \bigcup} B_{\varepsilon s}\Big( 2MT w_v(e) e  +\hat{c} (s-T) e\Big).
\ee
We show that $v(s,x)$ is small in $Q_s$ for $s\ge T$ by taking $T$ large. By Lemma~\ref{uv0}, the key step is to show that $\frac{1}{s} Q_s\subset \mathcal{C}_{v}^{\varepsilon}(e)\cup C$ for some compact set $C\subset \R^N\setminus\overline{\W_v}$.
Then, 
$$\frac{1}{s}Q_s=\underset{\hat{c}\ge w_v(e)+3\varepsilon}{ \bigcup} \Big(\Big[ \frac{2 MT w_v(e) }{s}+\hat{c}(1-\frac{T}{s})\Big] e +B_{\varepsilon}\Big).$$
For $s\in [T,MT]$, it follows from \eqref{epsilon} that
$$\frac{2MT w_v(e)}{s}+\hat{c}(1-\frac{T}{s})\ge 2w_v(e)> w_v(e)+2\varepsilon.$$ 
For $s\ge MT$, one also has 
$$\frac{2MT w_v(e)}{s}+\hat{c}(1-\frac{T}{s})\ge  \hat{c}(1-\frac{1}{M})>w_v(e)+2\varepsilon.$$
Then, $\frac{1}{s}Q_s$ is included in the $\varepsilon$-neighborhood of the ray $\{\tau (w_v(e)+2\varepsilon)e:\, \tau>1\}$. There is a compact set of $C\subset \R^N\setminus \overline{\W_{v}}$ such that $\frac{1}{s}Q_s \subset C\cup \mathcal{C}^{\varepsilon}_v(e)$ for all $s\in  [T, +\infty)$, where $\mathcal{C}^{\varepsilon}_v(e)$ is defined in Lemma~\ref{uv0}. By Lemma~\ref{uv0}, for any $\mu>0$, there is $T>0$  such that
\be\label{v-Qs} 
v(s,x)\le \mu, \quad \hbox{for $s\in  [T,+\infty)$ and $x\in Q_s$}.
\ee

Let 
\be\label{kappa} 
\kappa:=\min\Big(\frac{c}{4c'},\frac{c}{4Mw_v(e)}\Big).
\ee
Then, $0<c'-\varepsilon/4\le c'-c'\kappa\varepsilon/c<c'+c'\kappa\varepsilon/c\le c'+\varepsilon<c_u-\varepsilon$ by \eqref{epsilon}. So, $B_{c'\kappa\varepsilon/c}(c' e)\subset B_{c_u}\setminus \{0\}$.
Let $R>0$ and $\sigma_0>0$ such that Lemma~\ref{Rsigma} holds for $C=B_{c'\kappa\varepsilon/c}(c' e)$. Let $T>0$ such that $\varepsilon T\ge 2R$ and \eqref{v-Qs} holds for 
$$\mu:= \min\Big(\frac{1}{4a}\Big(1-\frac{(c_u-\varepsilon)^2}{c_u^2}\Big),\frac{1}{2a}\Big).$$
Notice that $|\xi|\le c_u-\varepsilon$ for any $\xi\in C$.
Since $s\ge T$ and by \eqref{epsilon}, \eqref{kappa}, it follows that
$$\varepsilon s-\frac{\kappa\varepsilon}{c}(2MTw_v(e)+c'(s-T))=\Big(1-\frac{c'\kappa }{c}\Big)\varepsilon s-\frac{\kappa\varepsilon}{c}(2MTw_v(e)-c'T)\ge \frac{1}{2}\varepsilon T\ge R.$$
This implies that 
$$B_{\frac{\kappa\varepsilon}{c}(2MTw_v(e)+c'(s-T))+R}(2MT w_v(e) e  +c' (s-T) e) \subset B_{\varepsilon s}\Big( 2MT w_v(e) e  + c' (s-T) e\Big)\subset Q_s.$$
Furthermore, for each $x_0\in B_{\frac{\kappa\varepsilon}{c}2MTw_v(e)}(2MT w_v(e) e )$, 
\begin{equation}
\begin{aligned}
x_0+C (s-T)+B_R&=x_0+B_{\frac{c'\kappa\varepsilon}{c}(s-T)}(c' (s-T)e)+B_R\\
&\subset B_{\frac{\kappa\varepsilon}{c}(2MTw_v(e)+c'(s-T))+R}(2MT w_v(e) e  +c' (s-T) e),
\end{aligned}
\end{equation}
where $C=B_{c'\kappa\varepsilon/c}(c' e)$.
Thus,  for each $x_0\in B_{\frac{\kappa\varepsilon}{c}2MTw_v(e)}(2MT w_v(e) e )$, it follows from \eqref{v-Qs} that
\be\label{v-Qs2}
v(s,x)\le \mu,\quad  \hbox{ for $s\in [T,+\infty)$ and $x\in x_0+C (s-T)+B_R$}.
\ee 
On the other hand, it follows from Lemma~\ref{delta-uv} that there is $\delta>0$ such that
$$u(T,x)\ge \delta,\, \hbox{ for all $x\in B_{\frac{\kappa\varepsilon}{c}2MTw_v(e)+R}(2MT w_v(e) e )$}.$$
Take $\sigma\in (0,\sigma_0]$ such that $\sigma e^{\frac{(c_u-\varepsilon)R}{2d}}\le \delta$, which makes 
$$u(T,x+x_0)\ge \underline{u}_{\hat{c},e'}(0,x), \hbox{ for each $x_0\in B_{\frac{\kappa\varepsilon}{c}2MTw_v(e)}(2MT w_v(e) e )$},$$
 for every $\hat{c}e'\in C=B_{c'\kappa\varepsilon/c}(c' e)$.
By \eqref{v-Qs2}, Lemma~\ref{Rsigma} and the comparison principle, it leads to that for any $\hat{c}e'\in C=B_{c'\kappa\varepsilon/c}(c' e)$ and $x_0\in B_{\frac{\kappa\varepsilon}{c}2MTw_v(e)}(2MT w_v(e) e )$, it holds
$$u(s+T,x+x_0)\ge\underline{u}_{\hat{c},e'}(s ,x) =\sigma e^{-\frac{\hat{c}}{2d}(x\cdot e'  - \hat{c}s) }\Phi_{R}\left( x -\hat{c} s e'  \right),$$
 for $s\ge 0$.  For any $t>\frac{2MT w_v(e)}{c}$, by taking $s= \frac{1}{c'}(ct -2MT w_v(e)) $, it follows from above inequality that there is $\delta'>0$ such that
\be\label{initial-s}
 u(\frac{1}{c'}(ct -2MT w_v(e))+T,x)\ge \delta',  \hbox{ for $x\in x_0+B_{\frac{R}{2}}(\frac{\hat{c}}{c'}(ct -2MT w_v(e)) e')$}.
\ee
with  any $\hat{c}e'\in C=B_{c'\kappa\varepsilon/c}(c' e)$ and $x_0\in B_{\frac{\kappa\varepsilon}{c}2MTw_v(e)}(2MT w_v(e) e )$. Notice that
\be
\begin{aligned} 
&\underset{\hat{c}e'\in B_{c'\kappa\varepsilon/c}(c' e)}{ \bigcup} \Big( B_{\frac{\kappa\varepsilon}{c}2MTw_v(e)}(2MT w_v(e) e )+B_{\frac{R}{2}}(\frac{\hat{c}}{c'}(ct -2MT w_v(e)) e') \Big) \\
=&\underset{\hat{c}e'\in B_{c'\kappa\varepsilon/c}(c' e)}{ \bigcup} \Big( B_{\frac{\kappa\varepsilon}{c}2MTw_v(e)+\frac{R}{2}}(2MT w_v(e) e +\frac{\hat{c}}{c'}(ct -2MT w_v(e)) e')\Big)\\
=&\underset{\hat{c}e'\in B_{c'\kappa\varepsilon/c}(c' e)}{ \bigcup} \Big(2MT w_v(e) e +\frac{1}{c'}(ct -2MT w_v(e)) \hat{c}e' + B_{\frac{\kappa\varepsilon}{c}2MTw_v(e)+\frac{R}{2}}\Big)\\
=&  2MT w_v(e) e +\frac{1}{c'}(ct -2MT w_v(e))B_{\frac{c'\kappa\varepsilon}{c}}(c' e)   + B_{\frac{\kappa\varepsilon}{c}2MTw_v(e)+\frac{R}{2}}\\
=&  2MT w_v(e) e + B_{ \kappa\varepsilon t -\frac{\kappa\varepsilon}{c}2MTw_v(e)}(cte -2MT w_v(e)e)   + B_{\frac{\kappa\varepsilon}{c}2MTw_v(e)+\frac{R}{2}}\\
=& B_{\kappa\varepsilon t +\frac{R}{2}}(cte).
\end{aligned}
\ee
Then, \eqref{initial-s} is equivalent to 
\be\label{initial-se}
u(\frac{1}{c'}(ct -2MT w_v(e))+T,x)\ge \delta',  \hbox{ for $t>2MTw_v(e)/c$ and $x\in B_{\kappa\varepsilon t +\frac{R}{2}}(cte)$}.
\ee
 
Let $t>2MT w_v(e)/c$ be fixed sufficiently large which will be given later precisely. For $s\in [\frac{1}{c'}(ct -2MT w_v(e))+T, t]$  and for any $x_0\in B_{\kappa\varepsilon t}(cte)$, define
$$\underline{u}(s,x)=\eta\Phi_{R/2}(x-x_0).$$
with $\eta:=\min(\delta',\frac{1}{2})$.
Then, by \eqref{initial-se},
$$\underline{u}(\frac{1}{c'}(ct -2MT w_v(e))+T,x)\le \eta\le u(\frac{1}{c'}(ct -2MT w_v(e))+T,x),  \hbox{ for $x\in B_{R/2}(x_0)$},$$
with any $x_0\in B_{\kappa\varepsilon t}(cte)$.
One can check that
\begin{eqnarray*}
\begin{array}{lll}
 N[\underline{u},v]&:=&(\underline{u})_s-d\Delta\underline{u}-r\underline{u}(1-\underline{u}-av)\\
&=&    -d \eta \Delta\Phi_{R/2} 
   -r \eta \Phi_{R/2} (1-\eta\Phi_{R/2} -a v) \\
&\le &   -\eta\Phi_{R/2} (r+d\lambda_{R/2}-r\eta^2-a r v) .
\end{array}
\end{eqnarray*}
For $x\not\in B_{R/2}(x_0)$, it is clear that $N[\underline{u},v]=0$. On the other hand, it is clear that $Q_s$ is the $\varepsilon s$-neighborhood of the ray $\{2MT w_v(e) e  +\hat{c} (s-T) e:\, \hat{c}\ge w_v(e)+3\varepsilon\}$ for $s\ge T$ and $\varepsilon T\ge 2R$. For $s\in [\frac{1}{c'}(ct -2MT w_v(e))+T, t]$, one can fix $t$ large enough such that $c/(2c')\le s/t\le 1$ and so, $\kappa\varepsilon t\le \varepsilon s/2$ by \eqref{kappa}.  Since $c>w_v(e)+3\varepsilon$, one can make $t$ even larger such that $ct>2MT w_v(e) e  +(w_v(e)+3\varepsilon)(s-T)$.
This implies that $B_{\kappa\varepsilon t +\varepsilon s/2}(cte)\subset Q_s$ and by \eqref{v-Qs}, one has 
\be\label{v-mu}
v(s,x)\le \mu\le \frac{1}{2a}, \hbox{ for $x\in B_{\varepsilon s/2}(x_0)$ and $s\in [\frac{1}{c'}(ct -2MT w_v(e))+T, +\infty)$},
\ee
with any $x_0\in B_{\kappa \varepsilon t}(cte)$.
Notice that $\varepsilon s\ge R$ for $s\ge T$ which means $B_{R/2}(x_0)\subset B_{\varepsilon s/2}(x_0)$.
By \eqref{v-mu}, the definition of $\eta$ and $\lambda_{R}\rightarrow 0$ as $R\rightarrow +\infty$, even if it means increasing $R$,  one can make $r+d\lambda_{R/2}-r\eta^2-a r v>0$ and so, $N[\underline{u},v]<0$ for $x\in B_{R/2}(x_0)$ and $s\in [\frac{1}{c'}(ct -2MT w_v(e))+T, t]$, with any $x_0\in B_{\kappa \varepsilon t}(cte)$. Therefore, $\underline{u}(s,x)$ is a subsolution of $u$ and hence $u(s,x)\ge \eta\Phi_{R/2}(x-x_0)$ for all $s\in [\frac{1}{c'}(ct -2MT w_v(e))+T, t]$. In conclusion, there is $\eta'>0$ such that for all sufficiently large $t$,
\be\label{initial-s2}
u(s,x)\ge \eta', \hbox{ for all $s\in [\frac{1}{c'}(ct -2MT w_v(e))+T, t]$ and $x\in B_{\kappa \varepsilon t+R/4}(cte)$}.
\ee

We claim that 
\be\label{e-Bce}
\inf_{x\in B_{\kappa\varepsilon t}(cte)} u(t,x)\rightarrow 1, \hbox{ as $t\rightarrow +\infty$}.
\ee
Assume by contradiction that there are sequences $\{t_n\}_{n\in\mathbb{N}}$ and $\{x_n\}_{n\in\mathbb{N}}$ such that $x_n\in  B_{\kappa\varepsilon t_n}(ct_n e)$, $t_n\rightarrow +\infty$ as $n\rightarrow +\infty$ and $u(t_n,x_n)\le 1-\eta''$ for some $\eta''\in (0,1)$. Let $u_n(t,x)=u(t+t_n,x+x_n)$ and so, $u_n(0,0)<1-\eta''$. By \eqref{initial-s2}, it follows that
$$u_n(t,x)\ge \eta', \hbox{ for $t\in [-t_n+\frac{1}{c'}(ct_n-2MTw_v(e))+T,0]$ and $x\in B_{R/4}$},$$
for large $n$. By \eqref{v-mu}, one has $v(t+t_n,x+x_n)\le \mu$ for $t\in [-t_n+\frac{1}{c'}(ct_n-2MTw_v(e))+T,+\infty)$ and $x\in B_{\varepsilon s_n}$ where $s_n=\frac{1}{c'}(ct_n-2MTw_v(e))+T$.
 By the standard parabolic estimates, $u_n(t,x)$ converge up to extraction of a subsequence to $u_{\infty}(t,x)$ satisfying
$$(u_{\infty})_t-d\Delta u_{\infty}-r u_{\infty}(1-u_{\infty}-a\mu)\ge 0, \hbox{ for $t\in (-\infty,+\infty)$ and $x\in\R^N$},$$
and
$$u_{\infty}(0,0)<1-\eta'',\quad u_{\infty}(t,x)\ge \eta' \hbox{ for $t\in (-\infty,0]$ and $x\in B_{R/4}$}.$$
By the classical results of Aronson and Weinberger \cite{Aro}, one knows that $u_{\infty}(t,x)\ge 1-a\mu$ for all $t\in\R$ and $x\in\R^N$. Since $\mu$ can be arbitrarily small, one reaches a contradiction by taking $\mu$ small such that $1-a\mu>1-\eta''$.

The conclusion of this lemma follows from \eqref{e-Bce} and \eqref{Bev0}.
\end{proof} 
 
In general, the projection of $e$ on $\R^+\mathcal{U}(U)\cup\{0\}$ is not $0$. So, we deal with more general situations than Lemma~\ref{lemma:bar-e=0} in the following lemma. The main idea of the proof has actually been displayed in Lemma~\ref{lemma:bar-e=0}, but it is more complicated.
 
 \begin{lemma}\label{lemma:eball-c}
Assume that assumptions of Theorem~\ref{Th1} hold. For every $e\in \mathbb{S}^{N-1}$ with a projection path $P(e,\R^+\mathcal{U}(U)\cup \{0\})$ such that $w_u(\widehat{\xi})>w_v(\widehat{\xi})$ for all $\xi\in P(e,\R^+\mathcal{U}(U)\cup \{0\})\setminus\{0\}$ and
$$\sup_{\xi\in P(e,\R^+\mathcal{U}(U)\cup \{0\})\setminus\{0\}}\frac{w_v(\widehat{\xi})\sqrt{1-(\widehat{\xi}\cdot \widehat{\overline{e}})^2}}{\sqrt{1-|\overline{e}|^2}}<c<w_u(e),$$ 
where $\overline{e}$ is the projection of $e$ on $\R^+\mathcal{U}(U)\cup \{0\}$, with the convention that $\hat{\overline{e}}=0$ if $\overline{e}=0$,
there is $\varepsilon>0$ such that 
\be\label{eq:eball-u}
\lim_{t\to +\infty}\sup_{x\in B_{\varepsilon}(ce)}\left \{ |u(t, xt)-1| + |v(t, xt)|\right \} =0.
\ee
\end{lemma}
 
 \begin{proof}
 If $\mathcal{U}(U)\neq \emptyset$ and $e\in \mathcal{U}(U)$, then $w_u(e)=w_{uv}(e)=+\infty$. For any $0\le c<+\infty$,there is $\varepsilon>0$ such that $B_{\varepsilon}(ce)\subset \W_{uv}$ by the definition of $\W_{uv}$. Then, \eqref{eq:eball-u} follows from Lemma~\ref{sycuv}.
 
 If $e\not\in\mathcal{U}(U)$, let $\overline{e}\in \R^+\mathcal{U}(U)\cup \{0\}$ be the projection of $e$ on $\R^+\mathcal{U}(U)\cup \{0\}$. Two cases may occur: (i) $\overline{e}=0$; (ii) $\overline{e}\neq 0$.
 
{\rm Case (i): $\overline{e}=0$.} The projection path $ P(e,\R^+\mathcal{U}(U)\cup \{0\})$ is the line connecting $0$ and $e$. So, $\hat{\xi}=e$ for all $\xi\in   P(e,\R^+\mathcal{U}(U)\cup \{0\})\setminus\{0\}$. By the definition of the projection path and since $\R^+\mathcal{U}(U)\cup \{0\}$ is a cone, either $\mathcal{U}(U)=\emptyset$ or $e\cdot \nu\le 0$ for all $\nu\in \mathcal{U}(U)$. By the definition of $w_u(e)$, one has $w_u(e)=c_u$. Thus, the condition of Lemma~\ref{lemma:eball-c} is equivalent to $w_v(\hat{\xi})=w_v(e)<w_u(e)=c_u$ for all $\xi\in   P(e,\R^+\mathcal{U}(U)\cup \{0\})\setminus \{0\}$. In this case, the conclusion is proved by Lemma~\ref{lemma:bar-e=0}.
 
{\rm Case (2): $\overline{e}\neq 0$.} 
In this case, the set $\mathcal{U}(U)$ is automatically not empty since $\overline{e}\in \R^+\mathcal{U}(U)$. In the poof of Lemma~\ref{sycuv}, we have shown that $\mathcal{U}(U)=\mathcal{U}(U_{\rho})$. It means that $\overline{e}\in \R^+\mathcal{U}(U_{\rho})$ as well. Let 
$$\mathcal{B}_1:=\{\hat{\xi}: \xi\in P(e,\R^+\mathcal{U}(U)\cup\{0\})\}.$$ 
By the definition of projection, one has that $\hat{\xi}\neq 0$, $\widehat{\xi} \in \mathbb{S}^{N-1}\setminus\mathcal{U}(U)=\mathcal{B}(U)$ for $\hat{\xi} \in \mathcal{B}_1\setminus\{\hat{\overline{e}}\}$ (remember that $\hat{\overline{e}}=\overline{e}/|\overline{e}|$)  
and $|e-\overline{e}|= \hbox{dist}(e,\R^+\mathcal{U}(U))= \inf_{\xi\in\mathcal{U}(U),\tau>0}|e-\tau\xi|$. By taking $\tau=|\overline{e}|$, one can show by simple computation that $e\cdot \widehat{\overline{e}}=|\overline{e}|\ge e\cdot \xi$ for all $\xi\in \mathcal{U}(U)$. Since $\widehat{\overline{e}}\in \mathcal{U}(U)$, this particularly implies 
\be\label{ov-e}
\sqrt{1-|\overline{e}|^2}=\sqrt{1-\Big(e\cdot \widehat{\overline{e}}\Big)^2}=\inf_{\nu\in\mathcal{U}(U),\nu\cdot e\ge 0}\sqrt{1-(e\cdot \nu)^2}.
\ee
Notice that $\overline{e}$ is also a projection of $\xi\in P(e,\R^+\mathcal{U}(U)\cup\{0\})\setminus\{\overline{e}\}$ on $\R^+\mathcal{U}(U)\cup \{0\}$ which also implies that $\xi\cdot \hat{\overline{e}}=|\overline{e}|$. Thus, as \eqref{ov-e}, one can show
\be\label{ov-e-xi}
\sqrt{1-\Big(\widehat{\xi}\cdot \widehat{\overline{e}}\Big)^2}=\inf_{\nu\in\mathcal{U}(U),\widehat{\xi}\cdot \nu\ge 0}\sqrt{1-(\widehat{\xi}\cdot \nu)^2} \quad \hbox{for every $\hat{\xi}\in \mathcal{B}_1\setminus\{\hat{\overline{e}}\}$}.
\ee
Then, by the definition of $w_u$, one gets 
\be\label{wuecu} 
w_u(\widehat{\xi})\sqrt{1-\Big(\widehat{\xi}\cdot \widehat{\overline{e}}\Big)^2}=c_u ,\,\hbox{ for every $\hat{\xi}\in \mathcal{B}_1\setminus\{\hat{\overline{e}}\}$},
\ee
and meantime, $w_u(\hat{\overline{e}})=+\infty$.

We claim that
\be\label{Path-lesswu}
\sup_{\hat{\xi}\in \mathcal{B}_1} w_v(\widehat{\xi})\sqrt{1-(\widehat{\xi}\cdot \widehat{\overline{e}})^2}<c_u.
\ee
Otherwise, since $\mathcal{B}_1$ is closed and by the continuity of $w_v(e)$, there is $\hat{\xi}_0\in \mathcal{B}_1$ such that 
$$w_v(\widehat{\xi_0})\sqrt{1-(\widehat{\xi_0}\cdot \widehat{\overline{e}})^2}\ge c_u.$$
If $\hat{\xi}_0=\hat{\overline{e}}$, then $0\ge c_u$ which is impossible. If $\hat{\xi}_0\neq \hat{\overline{e}}$, it follows from \eqref{wuecu} that 
$$c_u=w_u(\widehat{\xi_0})\sqrt{1-(\widehat{\xi_0}\cdot \widehat{\overline{e}})^2}\le w_v(\widehat{\xi_0})\sqrt{1-(\widehat{\xi_0}\cdot \widehat{\overline{e}})^2},$$
that is, $w_v(\widehat{\xi_0})\ge w_u(\widehat{\xi_0})$ which contradicts the condition of this lemma.

By \eqref{Path-lesswu}, one can take 
\be\label{eq:c<c'}
\sup_{\hat{\xi}\in \mathcal{B}_1}\frac{w_v(\widehat{\xi})\sqrt{1-(\widehat{\xi}\cdot \widehat{\overline{e}})^2}}{\sqrt{1-|\overline{e}|^2}}<c<c'<w_u(e).
\ee
Then,
$$\sup_{\hat{\xi}\in \mathcal{B}_1}w_v(\widehat{\xi})\sqrt{1-(\widehat{\xi}\cdot \widehat{\overline{e}})^2}< c\sqrt{1-|\overline{e}|^2}<c' \sqrt{1-|\overline{e}|^2}<w_u(e) \sqrt{1-|\overline{e}|^2}=c_u.$$
Take
\be\label{epsilon-2}
\begin{aligned}
0<\varepsilon<\min\Big(&1,\frac{1}{2}\sup_{\hat{\xi}\in \mathcal{B}_1}w_v(\widehat{\xi})\sqrt{1-(\widehat{\xi}\cdot \widehat{\overline{e}})^2},\frac{1}{3}(c\sqrt{1-|\overline{e}|^2}-\sup_{\hat{\xi}\in \mathcal{B}_1}w_v(\widehat{\xi})\sqrt{1-(\widehat{\xi}\cdot \widehat{\overline{e}})^2}),\\
&\frac{1}{12}(c'-c)(1-|\overline{e}|^2),\frac{c(c'-c)}{c'+3c}(1-|\overline{e}|^2),\frac{1}{3}(c_u-c'\sqrt{1-|\overline{e}|^2}),\\
&\frac{c|\overline{e}|}{(\sup_{\hat{\xi}\in \mathcal{B}_1}w_v(\widehat{\xi})+c_u)M'+1},\frac{c'-c}{2(c'+c)}\Big),
\end{aligned}
\ee
where
\be\label{eq:M'}
M'=\frac{2c+2c'+4}{c'-c}.
\ee
Since $\partial\W_v$ is continuous, even if it means decreasing $\varepsilon$, assume that $\overline{B_{\varepsilon}(ce)}\subset \R^N\setminus \overline{\W_v}$. By Lemma~\eqref{uv0}, one has 
\be\label{Bev0-2}
\lim_{t\rightarrow +\infty} \sup_{x\in B_{\varepsilon}(ce)} |v(t,xt)|=0
\ee

The condition that $w_v(\widehat{\xi})<w_u(\widehat{\xi})\le +\infty$ for all $\hat{\xi}\in \mathcal{B}_1$ implies $\widehat{\xi}\in \mathcal{B}(V)$ for all $\hat{\xi}\in \mathcal{B}_1$, that is, $\mathcal{B}_1\subset \mathcal{B}(V)$.
For $s\in [T, +\infty)$ with $T>0$, consider the set
$$\widetilde{Q}_s:=\underset{e\in \mathcal{B}_1}{ \bigcup} Q_s(e):= \underset{e\in \mathcal{B}_1}{ \bigcup}\, \underset{\hat{c}\ge w_v(e)+3\varepsilon}{ \bigcup} B_{\varepsilon s}\Big( 2MT w_v(e) e  +\hat{c} (s-T) e\Big),$$
where $M$ is sufficiently large such that $\hat{c}(1-\frac{1}{M})>w_v(e)+2\varepsilon$ for any $\hat{c}\ge w_v(e)+3\varepsilon$ and $e\in\mathcal{B}_1$.
As shown in the proof of Lemma~\ref{lemma:bar-e=0}, the set $\frac{1}{s}Q_s(e)$ is included in the $\varepsilon$-neighborhood of the ray $\{\tau (w_v(e)+2\varepsilon)e:\, \tau>1\}$ for every $e\in \mathcal{B}_1$ and $s\ge T$. Thus, there is a compact set  $C\subset \R^N\setminus \overline{\W_{v}}$ such that $\frac{1}{s}\widetilde{Q}_s \subset C\cup (\cup_{e\in\mathcal{B}_1} \mathcal{C}_v^{\varepsilon}(e))$ where $\mathcal{C}_v^{\varepsilon}(e)$ is defined in Lemma~\ref{uv0}, for all $s\in  [T, +\infty)$. By Lemma~\ref{uv1}, for any $\mu>0$, there is $T>0$  such that
\be\label{2-v-Qs} 
v(s,x)\le \mu, \quad \hbox{for $s\in  [T,+\infty)$ and $x\in \widetilde{Q}_s$}.
\ee

Since $\widehat{\overline{e}}\in \mathcal{U}(U)=\mathcal{U}(U_{\rho})$, it means that 
$$\lim_{t\rightarrow +\infty}\frac{\hbox{dist}(ct\overline{e},U_{\rho})}{ct|\overline{e}|}\rightarrow 0.$$
Thus, there are $y_t\in \R^N$ and $T_1>0$ such that $|y_t|\le  \varepsilon^2  t/2$ and
$$c t \overline{e} + y_t\in U_{\rho}, \hbox{ for all $t\ge T_1$}.$$
Let
$$\tilde{c}=\left|c(e-\overline{e})-\frac{y_t}{t}\right|,\, \tilde{c}'=\left|c'(e-\overline{e})-\frac{y_t}{t}\right|,\, \nu=\frac{ct(e-\overline{e})-y_t}{|ct(e-\overline{e})-y_t|}.$$
Here, although $\tilde{c}$, $\tilde{c}'$ and $\nu$ are depending on $t$, we drop $t$ for simplicity since we will fix $t$.
Then, by $|e-\overline{e}|=\sqrt{1-|\overline{e}|^2}$, 
$$c \sqrt{1-|\overline{e}|^2}-\varepsilon^2<\tilde{c}<c  \sqrt{1-|\overline{e}|^2}+\varepsilon^2 < c' \sqrt{1-|\overline{e}|^2}-\varepsilon^2<\tilde{c}'<c' \sqrt{1-|\overline{e}|^2}+\varepsilon^2<c_u-2\varepsilon,$$
for any $t\ge T_1$. By \eqref{epsilon-2}, one particularly has
\be\label{eq:c'+c}
\frac{c\tilde{c}'}{\tilde{c}}>\frac{c'+c}{2}. 
\ee
Let 
\be\label{2-kappa} 
\kappa:=\frac{\tilde{c}}{4\tilde{c}'}.
\ee
Then, $0<\tilde{c}'-\varepsilon/4<\tilde{c}'-\tilde{c}'\kappa \varepsilon<\tilde{c}'+\tilde{c}'\kappa\varepsilon/\tilde{c}\le \tilde{c}'+\varepsilon<c_u-\varepsilon$.
Let $R>0$ and $\sigma_0>0$ such that Lemma~\ref{Rsigma} holds for $C=B_{\tilde{c}'\kappa\varepsilon/\tilde{c}}(\tilde{c}' \nu)\subset B_{c_u}\setminus \{0\}$. Let $T>0$ such that $\varepsilon T\ge 4R$ and \eqref{2-v-Qs} holds for 
$$\mu\le \min\Big(\frac{1}{4a}\Big(1-\frac{(c_u-\varepsilon)^2}{c_u^2}\Big),\frac{1}{2a}\Big).$$
Notice that $|\xi|\le c_u-\varepsilon$ for any $\xi\in C$.

Fix any $t\ge T_1$ large enough such that
\be\label{eq:chosen-t}
\varepsilon t\ge T \hbox{ and }\varepsilon t (1-|\overline{e}|^2)>2MT\sup_{\hat{\xi}\in \mathcal{B}_1}w_v(\widehat{\xi}).
\ee
We prove that for $s\in [\varepsilon t, \tilde{c}t/\tilde{c}'+\varepsilon t]$
\be\label{BallinQs} 
\begin{aligned}
&B_{\frac{\tilde{c}'\kappa\varepsilon}{\tilde{c}} (s-\varepsilon t)+R}(ct\overline{e}+y_t  +\tilde{c}' (s-\varepsilon t) \nu)\\
=&B_{\frac{\tilde{c}'\kappa\varepsilon}{\tilde{c}} (s-\varepsilon t)+R}(ct\overline{e} +\tilde{c}' (s-\varepsilon t) \frac{ct(e-\overline{e})}{|ct(e-\overline{e})-y_t|})+y_t -\frac{\tilde{c}' (s-\varepsilon t) }{|ct(e-\overline{e})-y_t|}y_t \subset \widetilde{Q}_s.
\end{aligned}
\ee
Notice that 
$$\left|y_t -\frac{\tilde{c}' (s-\varepsilon t) }{|ct(e-\overline{e})-y_t|}y_t\right|=\left|1-\frac{\tilde{c}'(s-\varepsilon t)}{\tilde{c}t}\right| |y_t|\le |y_t|\le \frac{\varepsilon^2 t}{2}\le \frac{\varepsilon s}{2}.$$
and
\be\label{h-epsilon}
\frac{1}{2}\varepsilon s- \Big[\frac{\tilde{c}'\kappa\varepsilon}{\tilde{c}} (s-\varepsilon t)+R\Big]\ge \frac{1}{4}\varepsilon s -R\ge \frac{1}{4}\varepsilon T -R>0.
\ee
On the other hand, one has
$$\xi_s:=ct\overline{e} +\tilde{c}' (s-\varepsilon t) \frac{ct(e-\overline{e})}{|ct(e-\overline{e})-y_t|}=ct\left[ \Big(1-\frac{\tilde{c}'(s-\varepsilon t)}{|ct(e-\overline{e})-y_t|}\Big)\overline{e}+ \frac{\tilde{c}'(s-\varepsilon t)}{|ct(e-\overline{e})-y_t|} e\right].$$
Since 
$$0\le \frac{\tilde{c}'(s-\varepsilon t)}{|ct(e-\overline{e})-y_t|}=\frac{\tilde{c}'(s-\varepsilon t)}{\tilde{c}t}\le 1, \hbox{ for } s\in [\varepsilon t, \tilde{c}t/\tilde{c}'+\varepsilon t],$$
one has $\hat{\xi_s}\in\mathcal{B}_1$. Then, we only have to show 
\be\label{xis}
|\xi_s|\ge 2MT w_v(\hat{\xi_s})  + (w_v(\hat{\xi_s}) +3\varepsilon)  (s-T).
\ee
For $\varepsilon t\le s\le M'\varepsilon t$, one has 
$$|\xi_s|=\left|ct\overline{e} +\tilde{c}' (s-\varepsilon t) \frac{ct(e-\overline{e})}{|ct(e-\overline{e})-y_t|}\right|\ge ct|\overline{e}|,$$
by $\overline{e}\cdot (e-\overline{e})=0$,
and
\be
\begin{aligned} 
2MT w_v(\hat{\xi_s})  + (w_v(\hat{\xi_s}) +3\varepsilon)  (s-T)\le& \varepsilon t  + (w_v(\hat{\xi_s}) +3\varepsilon)  (M'\varepsilon t-T)\\
\le& \Big( (w_v(\hat{\xi_s}) +3\varepsilon) M' +1\Big)\varepsilon t\\
\le& \Big( (w_v(\hat{\xi_s}) +c_u) M' +1\Big)\varepsilon t\\
\le& c|\overline{e}| t,
\end{aligned}
\ee
by \eqref{epsilon-2}.
For $s\in [M'\varepsilon t,\tilde{c}t/\tilde{c}'+\varepsilon t]$, one has that
\be
\begin{aligned}
\xi_s\cdot (e-\overline{e})=&\frac{ct \tilde{c}'(s-\varepsilon t)}{|ct(e-\overline{e})-y_t|} (1-|\overline{e}|^2)=&\frac{c\tilde{c}'}{\tilde{c}}(s-\varepsilon t)(1-|\overline{e}|^2)\\
>&\frac{c'+c}{2}(s-\varepsilon t)(1-|\overline{e}|^2),
\end{aligned}
\ee
by \eqref{eq:c'+c},
and
\be
\begin{aligned} 
&\Big[2MT w_v(\hat{\xi_s})  + (w_v(\hat{\xi_s}) +3\varepsilon)  (s-T)\Big] \hat{\xi_s}\cdot (e-\overline{e})\\
=&\Big[2MT w_v(\hat{\xi_s})  + (w_v(\hat{\xi_s}) +3\varepsilon)  (s-T)\Big]\sqrt{1-|\overline{e}|^2} \sqrt{1-(\hat{\xi_s}\cdot \hat{\overline{e}})^2}\\
<& (c(1-|\overline{e}|^2)+3\varepsilon)s + (1-|\overline{e}|^2)\varepsilon t,
\end{aligned}
\ee
by \eqref{eq:c<c'} and \eqref{eq:chosen-t}.
In fact,
\be
\begin{aligned}
& \frac{c'+c}{2} (1-|\overline{e}|^2)(s-\varepsilon t)-(c(1-|\overline{e}|^2)+3\varepsilon)s - (1-|\overline{e}|^2)\varepsilon t\\
\ge& \frac{1}{4}(c'-c)(1-|\overline{e}|^2) s -\frac{c'+c}{2} (1-|\overline{e}|^2)   \varepsilon t- (1-|\overline{e}|^2)\varepsilon t\\
\ge& \frac{1}{4}(c'-c)(1-|\overline{e}|^2) M' \varepsilon t -\frac{c'+c}{2} (1-|\overline{e}|^2)   \varepsilon t- (1-|\overline{e}|^2)\varepsilon t=0,
\end{aligned}
\ee
by \eqref{eq:M'}.
In conclusion,
\eqref{xis} is proved and so \eqref{BallinQs} holds.
Thus, it follows from \eqref{2-v-Qs} that
\be\label{2-v-Qs2}
v(s,x+ct\overline{e}+y_t)\le \mu,\quad  \hbox{ for $s\in [\varepsilon t, \tilde{c}t/\tilde{c}']$ and $x\in C (s-\varepsilon t)+B_R$}.
\ee 
where $C=B_{\tilde{c}'\kappa\varepsilon/\tilde{c}}(\tilde{c}' \nu)$.

On the other hand, since $c t \overline{e} + y_t\in U_{\rho}$, it follows from \eqref{Bcuv} and Lemma~\ref{delta-uv} that there is $\delta>0$ such that
$$u(\varepsilon t,x)\ge \delta,\, \hbox{ for all $x\in B_{R}(ct \overline{e}+y_t)$ and $t\ge T/\varepsilon$}.$$
Take $\sigma\in (0,\sigma_0]$ such that $\sigma e^{\frac{(c_u-\varepsilon) R}{2d}}\le \delta$,  which makes 
$$u(\varepsilon t,x+ct \overline{e}+y_t)\ge \underline{u}_{\hat{c},e'}(0,x), \hbox{ for all $x\in\R^N$},$$
for every $\hat{c}e'\in C=B_{\tilde{c}'\kappa\varepsilon/\tilde{c}}(\tilde{c}'  \nu)$.
By \eqref{2-v-Qs2}, Lemma~\ref{Rsigma} and the comparison principle, it leads to that for any $\hat{c}e'\in C=B_{\tilde{c}'\kappa\varepsilon/\tilde{c}}(\tilde{c}'  \nu)$, it holds
$$u(s+\varepsilon t,x+ct \overline{e}+y_t)\ge\underline{u}_{\hat{c},e'}(s ,x) =\sigma e^{-\frac{\hat{c}}{2d}(x\cdot e'  - \hat{c}s) }\Phi_{R}\left( x -\hat{c} s e'  \right),$$
 for $s\in [\varepsilon t, \tilde{c}t/\tilde{c}']$.  By taking $s= \tilde{c}t/\tilde{c}'$, there is $\delta'>0$ such that
\be\label{2-initial-s} 
u(\frac{\tilde{c}t}{\tilde{c}'}+\varepsilon t,x+ct \overline{e}+y_t)\ge \delta', \hbox{ for $x\in B_{R/2}(\frac{\hat{c}\tilde{c} t}{\tilde{c}'}e')$},
\ee
 with any $\hat{c}e'\in C=B_{\tilde{c}'\kappa\varepsilon/\tilde{c}}(\tilde{c}'  \nu)$. Notice that
 \be
\begin{aligned} 
\underset{\hat{c}e'\in B_{\tilde{c}'\kappa\varepsilon/\tilde{c}}(\tilde{c}'  \nu)}{\bigcup} B_{\frac{R}{2}}(\frac{\hat{c}\tilde{c}t}{\tilde{c}'}e') =&  \underset{\hat{c}e'\in B_{\tilde{c}'\kappa\varepsilon/\tilde{c}}(\tilde{c}'  \nu)}{\bigcup} \left(\frac{\tilde{c}t}{\tilde{c}'}\hat{c}e' +B_{\frac{R}{2}}\right)\\
=& \frac{\tilde{c}t}{\tilde{c}'} B_{\tilde{c}'\kappa\varepsilon/\tilde{c}}(\tilde{c}'  \nu) +B_{\frac{R}{2}} \\
=& B_{\kappa\varepsilon t +\frac{R}{2}}(\tilde{c} t\nu).
\end{aligned}
\ee
Then, by the definition of $\nu$, one has that $\tilde{c} t \nu +ct \overline{e} +y_t=cte$ and \eqref{2-initial-s} is equivalent to
\be\label{2-initial-se}
u(\frac{\tilde{c}t}{\tilde{c}'}+\varepsilon t,x)\ge \delta', \hbox{ for large $t$ satisfying \eqref{eq:chosen-t} and $x\in B_{\kappa\varepsilon t +\frac{R}{2}}(cte)$},
\ee

For $s\in [\tilde{c}t/\tilde{c}' +\varepsilon t, t]$  and  any $x_0\in B_{\kappa\varepsilon t}(cte)$, define
$$\underline{u}(s,x)=\eta\Phi_{R/2}(x-x_0).$$
with $\eta=\min(\delta',\frac{1}{2})$.
Then, by \eqref{2-initial-se},
$$\underline{u}(\frac{\tilde{c}t}{\tilde{c}'},x)\le \eta\le u(\frac{\tilde{c}t}{\tilde{c}'},x),  \hbox{ for $x\in B_{R/2}(x_0)$},$$
with any $x_0\in B_{\kappa\varepsilon t}(cte)$.
One can check that
\begin{eqnarray*}
\begin{array}{lll}
 N[\underline{u},v]&:=&(\underline{u})_s-d\Delta\underline{u}-r\underline{u}(1-\underline{u}-av)\\
&=&    -d \eta \Delta\Phi_{R/2} 
   -r \eta \Phi_{R/2} (1-\eta\Phi_{R/2} -a v) \\
&\le &   -\eta\Phi_{R/2} (r+d\lambda_{R/2}-r\eta^2-a r v) .
\end{array}
\end{eqnarray*}
For $x\not\in B_{R/2}(x_0)$, it is clear that $N[\underline{u},v]=0$. On the other hand, it is clear that $\widetilde{Q}_s$ contains the $\varepsilon s$-neighborhood of the ray $\{2MT w_v(e) e  +\hat{c} (s-T) e:\, \hat{c}\ge w_v(e)+3\varepsilon\}$ for $s\ge T$ and $\varepsilon T\ge 4R$. For  $s\in [\tilde{c}t/\tilde{c}' +\varepsilon t, t]$, one has that $\kappa\varepsilon t\le \varepsilon s/2$ by \eqref{2-kappa}.  Since $c>w_v(e)+3\varepsilon$, one can make $t$ even larger such that $ct>2MT w_v(e) e  +(w_v(e)+3\varepsilon)(s-T)$. This implies that $B_{\kappa\varepsilon t +\varepsilon s/2}(cte)\subset Q_s$ and by \eqref{2-v-Qs}, one has 
\be\label{2-v-mu}
v(s,x)\le \mu\le \frac{1}{2a}, \hbox{ for $x\in B_{\varepsilon s/2}(x_0)$ and  $s\in [\tilde{c}t/\tilde{c}' +\varepsilon t, t]$},
\ee
with any $x_0\in B_{\kappa \varepsilon t}(cte)$.
Notice that $\varepsilon s\ge R$ for $s\ge T$ which means $B_{R/2}(x_0)\subset B_{\varepsilon s/2}(x_0)$.
By \eqref{2-v-mu}, the definition of $\eta$ and $\lambda_{R}\rightarrow 0$ as $R\rightarrow +\infty$, even if it means increasing $R$,  one can make $r+d\lambda_{R/2}-r\eta^2-a r v>0$ and so $N[\underline{u},v]<0$ for $x\in B_{R/2}(x_0)$ and  $s\in [\tilde{c}t/\tilde{c}' +\varepsilon t, t]$, with any $x_0\in B_{\kappa \varepsilon t}(cte)$. Therefore, $\underline{u}(s,x)$ is a subsolution of $u$ and hence $u(s,x)\ge \eta\Phi_{R/2}(x-x_0)$ for  $s\in [\tilde{c}t/\tilde{c}' +\varepsilon t, t]$. In conclusion, there is $\eta'>0$ such that for all sufficiently large $t$,
\be\label{2-initial-s2}
u(s,x)\ge \eta', \hbox{ for  $s\in [\tilde{c}t/\tilde{c}' +\varepsilon t, t]$ and $x\in B_{\kappa \varepsilon t+R/4}(cte)$}.
\ee

We claim that 
\be\label{2-e-Bce}
\inf_{x\in B_{\kappa\varepsilon t}(cte)} u(t,x)\rightarrow 1, \hbox{ as $t\rightarrow +\infty$}.
\ee
Assume by contradiction that there are sequences $\{t_n\}_{n\in\mathbb{N}}$ and $\{x_n\}_{n\in\mathbb{N}}$ such that $x_n\in  B_{\kappa\varepsilon t_n}(ct_n e)$, $t_n\rightarrow +\infty$ as $n\rightarrow +\infty$ and $u(t_n,x_n)\le 1-\eta''$ for some $\eta''\in (0,1)$. Let $u_n(t,x)=u(t+t_n,x+x_n)$ and so, $u_n(0,0)<1-\eta''$. By \eqref{2-initial-s2}, it follows that
$$u_n(t,x)\ge \eta', \hbox{ for $t\in [-t_n+\tilde{c}t_n/\tilde{c}' +\varepsilon t_n,0]$ and $x\in B_{R/4}$},$$
for large $n$. By \eqref{2-v-mu}, one has $v(t+t_n,x+x_n)\le \mu$ for $t\in [-t_n+\tilde{c}t_n/\tilde{c}' +\varepsilon t_n,0]$ and $x\in B_{\varepsilon s_n}$ where $s_n=\tilde{c}t_n/\tilde{c}' +\varepsilon t_n$.
 By the standard parabolic estimates, $u_n(t,x)$ converge up to extraction of a subsequence to $u_{\infty}(t,x)$ satisfying
$$(u_{\infty})_t-d\Delta u_{\infty}-r u_{\infty}(1-u_{\infty}-a\mu)\ge 0, \hbox{ for $t\in (-\infty,0]$ and $x\in\R^N$},$$
and
$$u_{\infty}(0,0)<1-\eta'',\quad u_{\infty}(t,x)\ge \eta' \hbox{ for $t\in (-\infty,0]$ and $x\in B_{R/4}$}.$$
By the classical results of Aronson and Weinberger \cite{Aro}, one knows that $u_{\infty}(t,x)\ge 1-a\mu$ for all $t\in (-\infty,0]$ and $x\in\R^N$. Since $\mu$ can be arbitrarily small, one reaches a contradiction by taking $\mu$ small such that $1-a\mu>1-\eta''$.

The conclusion of this lemma follows from \eqref{2-e-Bce} and \eqref{Bev0-2}.
\end{proof}


Now, we remove the restriction of $c$ in Lemma~\ref{lemma:eball-c}.

 \begin{lemma}\label{lemma:wu>wv}
Assume that assumptions of Theorem~\ref{Th1} hold. For every $e\in \mathbb{S}^{N-1}$ with a projection path $P(e,\R^+\mathcal{U}(U)\cup \{0\})$ such that $w_u(\widehat{\xi})>w_v(\widehat{\xi})$ for all $\xi\in P(e,\R^+\mathcal{U}(U)\cup \{0\})\setminus\{0\}$ and $0\le c<w_u(e)$,
there is $\varepsilon>0$ such that 
$$\lim_{t\to +\infty}\sup_{x\in  B_{\varepsilon}(ce)}\left \{ |u(t, xt)-1| + |v(t, xt)|\right \} =0.$$
\end{lemma}

\begin{proof}
For any $0\le c<w_{uv}(e)$, there is $\varepsilon>0$ such that $\overline{B_{\varepsilon}(ce)}\subset \W_{uv}$. Because $\W_{uv}$ is a spreading subset of $u$ and $\W_v\setminus \overline{\W_{uv}}$ is a spreading superset of $v$ by Lemma~\ref{sycuv}, the conclusion follows immediately.

For any $w_{uv}(e)\le c<w_u(e)$, we only have to deal with the case
$$c\le \sup_{\xi\in P(e,\R^+\mathcal{U}(U)\cup \{0\})\setminus\{0\}}\frac{w_v(\widehat{\xi})\sqrt{1-(\widehat{\xi}\cdot \widehat{\overline{e}})^2}}{\sqrt{1-|\overline{e}|^2}},$$
by Lemma~\ref{lemma:eball-c}. Take
$$\sup_{\xi\in P(e,\R^+\mathcal{U}(U)\cup \{0\})\setminus\{0\}}\frac{w_v(\widehat{\xi})\sqrt{1-(\widehat{\xi}\cdot \widehat{\overline{e}})^2}}{\sqrt{1-|\overline{e}|^2}}<c_1<w_u(e).$$
Then, $c_1>c$. By Lemma~\ref{lemma:eball-c}, there is $\varepsilon>0$ such that 
\be\label{2epsilon}
\lim_{t\to +\infty}\sup_{x\in B_{\frac{c_1}{c}2\varepsilon}(c_1 e)}\left \{ |u(t, xt)-1| + |v(t, xt)|\right \} =0.
\ee
Let $\rho>0$ and $\delta\in (0,1)$ be given by Lemma~\ref{lemma:sub}. 
Then, for $t>0$ large enough , it follows from \eqref{2epsilon} that
$$u(\frac{c}{c_1}t,x)\ge 1-\delta,\, v(\frac{c}{c_1}t,x)\le \delta \hbox{ for $x\in B_{2\varepsilon t}(c t e)$}.$$
Even if it means increasing $t$, one can assume that $\varepsilon t>\rho$, that is, $B_{\varepsilon  t+\rho}(c t e)\subset B_{2\varepsilon t}(c t e)$. Then, for any $x_0\in B_{\varepsilon t}(cte)$,
$$u(\frac{c}{c_1}t+s,x_0+x)\ge \underline{u}_{\rho}(s,x),\, v(\frac{c}{c_1}t+s,x_0+x)\le \overline{v}_{\rho}(s,x), \hbox{ for $s\ge 0$ and $x\in\R^N$},$$
where $(\underline{u}_{\rho},\overline{v}_{\rho})(t,x)$ is the solution of \eqref{eq-uv-rho}.
 By Lemma~\ref{lemma:sub}, one has that
$$u(t,x_0)\ge 1-2\delta,\, v(t,x_0)\le 2\delta \hbox{ for any $x_0\in B_{\varepsilon t}(c t e)$}.$$
By arbitrariness of $\delta$, the conclusion follows.
\end{proof}

Symmetrically to Lemma~\ref{lemma:eball-c}, we have the following lemma for $v$ by applying Lemma~\ref{Lemma:Rsigma-v}.

 \begin{lemma}\label{lemma:eball-v1}
Assume that the assumptions of Theorem~\ref{Th1} hold. For every $e\in \mathbb{S}^{N-1}$ with a projection path $P(e,\R^+\mathcal{U}(V)\cup \{0\})$ such that $w_v(\widehat{\xi})>w_u(\widehat{\xi})$ for all $\xi\in P(e,\R^+\mathcal{U}(V)\cup \{0\})\setminus\{0\}$ and
$$\sup_{\xi\in P(e,\R^+\mathcal{U}(V)\cup \{0\})\setminus\{0\}}\frac{w_u(\widehat{\xi})\sqrt{1-(\widehat{\xi}\cdot \widehat{\overline{e}})^2}}{\sqrt{1-|\overline{e}|^2}}<c<w_v(e),$$ 
where $\overline{e}$ is the projection of $e$ on $\R^+\mathcal{U}(V)\cup \{0\}$, with the convention that $\hat{\overline{e}}=0$ if $\overline{e}=0$,
there is $\varepsilon>0$ such that 
$$\lim_{t\to +\infty}\sup_{x\in B_{\varepsilon}(ce)}\left \{ |u(t, xt)| + |v(t, xt)-1|\right \} =0.$$
\end{lemma}

To remove the restriction of $c$ in Lemma~\ref{lemma:eball-v1}, we introduce another tool in the paper of Aronson and Weinberger \cite{Aro}.

\begin{lemma}\label{lemm3}
	For any $0<\delta<1/b$ and  any $c\in (0, c_{\delta})$
	($c_{\delta}=2\sqrt{1-b\delta}$), there exist
	$\beta(c)<1-b\delta$ such that for any $\beta\in [\beta(c),1-b\delta)$
	the solution of
	\begin{equation}
		\begin{cases}
			\hat{V}''+c\hat{V}'+\hat{V}(1-\hat{V}-b\delta)=0,\\
			\hat{V}(0)=\beta, ~~\hat{V}'(0)=0
		\end{cases}
	\end{equation}
	satisfies $\hat{V}(R)=0$ for some $R>0$ and $\hat{V}'(\xi)<0$ for
	$\xi \in (0, R]$.
\end{lemma}

 \begin{lemma}\label{lemma:eball-v}
Assume that the assumptions of Theorem~\ref{Th1} hold. For every $e\in \mathbb{S}^{N-1}$ with a projection path $P(e,\R^+\mathcal{U}(V)\cup \{0\})$ such that $w_v(\widehat{\xi})>w_u(\widehat{\xi})$ for all $\xi\in P(e,\R^+\mathcal{U}(V)\cup \{0\})\setminus\{0\}$ and $w_u(e)<c<w_v(e)$ ,
there is $\varepsilon>0$ such that 
$$\lim_{t\to +\infty}\sup_{x\in B_{\varepsilon}(ce)}\left \{ |u(t, xt)| + |v(t, xt)-1|\right \} =0.$$
\end{lemma}

\begin{proof}
Take any $\delta\in (0,1/b)$ and $c'\in (0,c_{\delta})$.  Take any $\beta\in [\beta(c'),1-b\delta)$. Let $\hat{V}$ be given by Lemma~\ref{lemm3} for $\delta$, $c'$ and $\beta$. For $t\ge 0$ and $x\in\R^N$, define
\begin{equation}
		\underline{v}(t, x) =
		\begin{cases}
			\beta & \text{if } |x|<\rho,\\
			\hat{V}(|x|-\rho) & \text{if } \rho\le
			|x|\le \rho+R,\\
			0 & \text{if } |x| > \rho+R,
		\end{cases}
	\end{equation}
	where $\rho$ is a large constant to be chosen.
For $\rho\le |x|\le \rho+R$, one can check that
\begin{equation}
\begin{aligned}
\underline{v}_t-\Delta \underline{v}-\underline{v}(1-\underline{v}-b\delta)=-\hat{V}''-\hat{V}'\frac{N-1}{|x|}-\hat{V}(1-\hat{V}-b\delta)=\Big(c'-\frac{N-1}{|x|}\Big)\hat{V}'\le 0,
\end{aligned}
\end{equation}
for $\rho$ sufficiently large. Obviously, such an inequality holds for $|x|<\rho$ and $|x|>\rho+R$.

For any $w_{u}(e)\le c<w_v(e)$, we only have to deal with the case
$$c\le \sup_{\xi\in P(e,\R^+\mathcal{U}(V)\cup \{0\})\setminus\{0\}}\frac{w_u(\widehat{\xi})\sqrt{1-(\widehat{\xi}\cdot \widehat{\overline{e}})^2}}{\sqrt{1-|\overline{e}|^2}},$$
by Lemma~\ref{lemma:eball-v1}. Take
$$\sup_{\xi\in P(e,\R^+\mathcal{U}(V)\cup \{0\})\setminus\{0\}}\frac{w_u(\widehat{\xi})\sqrt{1-(\widehat{\xi}\cdot \widehat{\overline{e}})^2}}{\sqrt{1-|\overline{e}|^2}}<c_1<w_v(e).$$
Then, $c_1>c$. By Lemma~\ref{lemma:eball-v1}, there is $\varepsilon>0$ such that 
\be\label{2epsilon-v}
\lim_{t\to +\infty}\sup_{x\in B_{\frac{c_1}{c}2\varepsilon}(c_1 e)}\left \{ |u(t, xt)| + |v(t, xt)-1|\right \} =0.
\ee
By the continuity of $\partial\W_u$, one can decreasing $\varepsilon$ such that 
$$\overline{\underset{1\le \tau\le c_1/c}{\cup} B_{(\tau+1) \varepsilon}(\tau c e)}\subset \R^N\setminus \overline{\W_u}.$$
Then, $\frac{1}{s}B_{\varepsilon t+s\varepsilon}(cte)\subset \underset{1\le \tau\le c_1/c}{\cup} B_{(\tau+1) \varepsilon}(\tau c e)$ for all $s\in [ct/c_1,t]$. Since $\W_u$ is a spreading superset of $u$, there is $T>0$ such that 
$$u(s,x)\le \delta,  \hbox{ for $t\ge T$, $s\in [ct/c_1,t]$ and $x\in B_{\varepsilon t}(c t e)$}.$$
By \eqref{2epsilon-v}, even if it means increasing $T$, it holds that
$$v(\frac{c}{c_1}t,x)\ge \beta \hbox{ for $t\ge T$ and $x\in B_{2\varepsilon t}(c t e)$},$$
and $2\varepsilon t\ge \varepsilon t+\rho+R$, that is, $B_{\varepsilon  t+\rho+R}(c t e)\subset B_{2\varepsilon t}(c t e)$ for $t\ge T$. It implies that for any $x_0\in B_{\varepsilon t}(c t e)$,
$$v(\frac{c}{c_1}t,x_0+x)\ge \underline{v}(0,x), \hbox{ for all $x\in\R^N$}.$$
Then, $v(\frac{c}{c_1}t+s,x_0+x)\ge \underline{v}(s,x)$ for  $s\in [0,t-ct/c_1]$ with any $x_0\in B_{\varepsilon t}(cte)$.
One has that
$$v(t,x)\ge \beta \hbox{ for $x\in B_{\varepsilon t}(c t e)$}.$$
By arbitrariness of $\beta$ and $\delta$, the conclusion follows.
\end{proof}

Now, we are ready to prove Theorem~\ref{Th1}.

\begin{proof}[Proof of Theorem~\ref{Th1}]
We first show \eqref{Sspeed-u}. By Lemma~\ref{lemma:wu>wv}, for any $0<c'\le c<w_u(e)$, there is $\varepsilon(c')>0$ such that
$$\lim_{t\to +\infty}\sup_{x\in  B_{\varepsilon(c')}(c'e)}\left \{ |u(t, xt)-1| + |v(t, xt)|\right \} =0.$$
Since the line segment $\{se:\, 0\le s\le c\}$ can be covered by finite balls $B_{\varepsilon(c_i)}(c_i e)$ with $0\le c_i\le c$, the first line of \eqref{Sspeed-u} is proved. The second line of \eqref{Sspeed-u} is proved by Lemma~\ref{uv0} and the ray $\{se:\, s\ge c\}\subset \mathcal{C}_u^{\varepsilon}(e)$ with $\varepsilon<c-w_u(e)$. For the third line of \eqref{Sspeed-u}, we know from \eqref{Pathwu>wv} that $w_u(e)>w_v(e)$. For $w_v(e)<c<w_u(e)$, the first line of \eqref{Sspeed-u} has shown that
$$\lim_{t\rightarrow +\infty}\sup_{0\le s\le c}|v(t,ste)|=0.$$
Meantime, by Lemma~\ref{uv0} and $\{se:\,s\ge c\}\subset \mathcal{C}_v^{\varepsilon}(e)$ with $\varepsilon<c-w_v(e)$, one has 
$$\lim_{t\rightarrow +\infty}\sup_{s\ge c}|v(t,ste)|=0.$$
Then, the third line of \eqref{Sspeed-u} follows.

Then, we show \eqref{Sspeed-v}. For any $0\le c<w_{uv}(e)$, the line segment $\{se:\, 0\le s\le c\}$ is a compact subset of $\W_{uv}$. The first line of \eqref{Sspeed-v} follows from Lemma~\ref{sycuv}. For any $w_u(e)<c_1<c_2<w_v(e)$, the line segment $\{se:\, c_1\le s\le c_2\}$ is a compact subset of $\W_v\setminus \overline{\W_u}$. The second line of \eqref{Sspeed-v} follows from Lemma~\ref{lemma:eball-v}. For $c>w_v(e)>w_u(e)$, $\{se:\,s\ge c\}\subset \mathcal{C}_v^{\varepsilon}(e)$ and $\{se:\,s\ge c\}\subset \mathcal{C}_u^{\varepsilon}(e)$ with $\varepsilon<c-w_v(e)$. The third line of \eqref{Sspeed-v} follows from Lemma~\ref{uv0}.
\end{proof}

\subsection{Proof of Theorem~\ref{Th2}}
In this subsection, we prove Theorem~\ref{Th2}. By Definitions~\ref{def:Path}-\ref{def:starshape},  if $\R^+(\W_u\setminus \overline{\W_v})\cup \{0\}$  and $\R^+(\W_v\setminus \overline{\W_u})\cup \{0\}$ are star-shaped with respect to $\R^+\mathcal{U}(U)\cup \{0\}$ and $\R^+\mathcal{U}(V)\cup \{0\}$ respectively, then $P(e,\R^+\mathcal{U}(U)\cup \{0\})\subset \R^+(\W_u\setminus \overline{\W_v})\cup \{0\}$ for every $e\in\mathbb{S}^{N-1}$ such that $w_u(e)>w_v(e)$ and $P(e,\R^+\mathcal{U}(V)\cup \{0\})\subset \R^+(\W_v\setminus \overline{\W_u})\cup \{0\}$ for every $e\in\mathbb{S}^{N-1}$ such that $w_u(e)<w_v(e)$ which implies \eqref{Pathwu>wv} and \eqref{Pathwu<wv}. Therefore, \eqref{Starshape} is stronger than \eqref{Pathwu>wv} and \eqref{Pathwu<wv}. So, all lemmas in Section~\ref{Section1.7} are available under conditions of Theorem~\ref{Th2}.

We first show a byproduct of Lemma~\ref{lemma:wu>wv} and Lemma~\ref{lemma:sub}.

 \begin{lemma}\label{Su-cone}
Assume that the assumptions of Theorem~\ref{Th1} hold. For every $e\in \mathbb{S}^{N-1}$ with a projection path $P(e,\R^+\mathcal{U}(U)\cup \{0\})$ such that $w_u(\widehat{\xi})>w_v(\widehat{\xi})$ for all $\xi\in P(e,\R^+\mathcal{U}(U)\cup \{0\})\setminus\{0\}$, $c_{uv}<c<w_u(e)$ and any $\varepsilon\in (0,c_{uv})$,
it holds that 
$$\lim_{t\to +\infty}\sup_{x\in \underset{0< \tau< 1}{\cup}B_{\tau (c_{uv}-\varepsilon)}((1-\tau) ce)}\left \{ |u(t, xt)-1| + |v(t, xt)|\right \} =0.$$
\end{lemma}

\begin{proof}
Since $\overline{B_{c_{uv}-\varepsilon/2}}$ is a compact subset of $\W_{uv}$ and $B_{\tau (c_{uv}-\varepsilon)}((1-\tau) ce)\subset \overline{B_{c_{uv}-\varepsilon/2}}$ for all $\tau\ge 1-\varepsilon/
(2c-2c_{uv}+2\varepsilon)$, it follows from Lemma~\ref{sycuv} that
\be\label{small-tau}
\lim_{t\to +\infty}\sup_{x\in \underset{1-\varepsilon/(2c-2c_{uv}+2\varepsilon)\le  \tau< 1}{\cup}B_{\tau (c_{uv}-\varepsilon)}((1-\tau) ce)}\left \{ |u(t, xt)-1| + |v(t, xt)|\right \} =0.
\ee

From Lemma~\ref{lemma:wu>wv}, one has that 
$$\lim_{t\to +\infty}\sup_{x\in  B_{ \varepsilon}(ce)}\left \{ |u(t, xt)-1| + |v(t, xt)|\right \} =0.$$
Thus, for any $0<\tau\le 1-\varepsilon/
(2c-2c_{uv}+2\varepsilon)$ and $\delta>0$, there is $T>0$ independent of $\tau$ such that $\varepsilon (1-\tau)T\ge \rho$ with $\rho>0$ given by Lemma~\ref{lemma:sub} and
$$u((1-\tau)t,x)\ge 1-\delta,\, v((1-\tau)t,x)\le \delta, \hbox{ for $t\ge T$ and $x\in B_{\rho}((1-\tau)c t e)$}.$$
Then, 
$$u((1-\tau)t+s,(1-\tau)cte+x)\ge \underline{u}_{\rho}(s,x),\, v((1-\tau)t+s,(1-\tau)cte+x)\le \overline{v}_{\rho}(s,x),$$ 
for $s\ge 0$ and $x\in\R^N$,
where $(\underline{u}_{\rho},\overline{v}_{\rho})$ is the solution of \eqref{eq-uv-rho}.
By Lemma~\ref{lemma:sub}, one has that
$$u(t,x)\ge 1-2\delta,\, v(t,x)\le 2\delta, \hbox{ for $t\ge T$ and $x\in B_{(c_{uv}-\varepsilon) \tau t }((1-\tau)c t e)$}.$$
By arbitrariness of $\delta$, one gets that
\be\label{large-tau}
\lim_{t\to +\infty}\sup_{x\in \underset{0< \tau\le 1-\varepsilon/
(2c-2c_{uv}+2\varepsilon)}{\cup}B_{\tau (c_{uv}-\varepsilon)}((1-\tau) ce)}\left \{ |u(t, xt)-1| + |v(t, xt)|\right \} =0.
\ee

The conclusion follows from \eqref{small-tau} and \eqref{large-tau}.
\end{proof}

Lemma~\ref{Su-cone} has already implied that the set $\mathcal{S}_u$ is a spreading subset of $u$. By the definition of $\mathcal{S}_u$, one knows that $\W_u\setminus \overline{\mathcal{S}_u}\subset  \W_v\setminus \overline{\mathcal{S}_u}=\mathcal{S}_v$. Since $\W_u$ is a spreading superset of $u$, we only have to show that $\mathcal{S}_v$ is a spreading subset of $v$.

\begin{lemma}\label{Sv-cone}
Assume that all conditions of Theorem~\ref{Th2} hold. For every $ce\in \mathcal{S}_v$,
there is $\varepsilon>0$ such that 
$$\lim_{t\to +\infty}\sup_{x\in B_{\varepsilon}(ce)}\left \{ |u(t, xt)| + |v(t, xt)-1|\right \} =0.$$
\end{lemma}

\begin{proof}
Take any $ce\in \W_u\setminus \overline{\mathcal{S}_u}$. Then, since $\W_u\setminus \overline{\mathcal{S}_u}$ is an open set, there is $\varepsilon>0$ such that $ce\not\in \mathcal{S}_u +B_{2\varepsilon}$.
We claim that 
\be\label{cone-inSv}
\Big(ce+\underset{\tau>0}{\cup} B_{\tau(c_{uv}+\varepsilon)}(\tau ce)\Big)\cap (\mathcal{S}_u+B_{\varepsilon})=\emptyset.
\ee
Assume by contradiction that  there are $\tau>0$, $\tau'\in [0,\tau)$ and $\xi\in\mathbb{S}^{N-1}$ such that
$$(1+\tau)ce +\tau' (c_{uv}+\varepsilon)\xi\in ( ce+ \underset{\tau>0}{\cup} B_{\tau (c_{uv}+\varepsilon)}(\tau ce)) \cap (\mathcal{S}_u+B_{\varepsilon}).$$
Notice that for any $c'\nu \in \mathcal{S}_u+B_{\varepsilon}$, one has that $(c'-\varepsilon)\nu \in \mathcal{S}_u$ and
\be\label{e-cone}
 \underset{0<\tau<1}{\cup} B_{\tau (c_{uv}+\varepsilon)}((1-\tau)c' \nu)\subset  \underset{0<\tau<1}{\cup} B_{\tau c_{uv}}((1-\tau)(c'-\varepsilon) \nu)+B_{\varepsilon}\subset \mathcal{S}_u+B_{\varepsilon},
 \ee
by the definition of $\mathcal{S}_u$.
Then it follows  that
$$ce\in \overline{B_{\frac{\tau}{1+\tau} (c_{uv}+\varepsilon)}\Big(ce+\frac{\tau'}{1+\tau} (c_{uv}+\varepsilon) \xi\Big)}\subset \overline{\mathcal{S}_u+B_{\varepsilon}}\subset \mathcal{S}_u+B_{2\varepsilon},$$
which is a contradiction.

By \eqref{cone-inSv}, the definition of $\mathcal{S}_u$ and \eqref{C-A}, we have that 
$$w_v(\hat{\xi})>w_u(\hat{\xi}),  \hbox{ for any }\xi\in \Big(ce+\underset{\tau>0}{\cup} B_{\tau(c_{uv}+\varepsilon)}(\tau ce)\Big)\cap \W_u.$$
In other words,
$$\overline{\Big(ce+\underset{\tau>0}{\cup} B_{\tau(c_{uv}+\varepsilon)}(\tau ce)\Big)\cap \W_u}\subset \W_v$$
 By continuity of $\partial\W_u$ and $\partial\W_v$, there is $\mu>0$ such that 
\be\label{cone-inWv}
\overline{\Big(ce +\underset{\tau>0}{\cup} B_{\tau(c_{uv}+\varepsilon)}(\tau ce)\Big)\cap (\W_u+B_{2\mu})}\subset \W_v.
\ee
Assume that $\mu>\varepsilon$ even if it means decreasing $\varepsilon$.
Denote
$$\mathcal{H}_1:=\Big(ce+\underset{\tau>0}{\cup} B_{\tau(c_{uv}+\varepsilon)}(\tau ce)\Big)\cap (\W_u+B_{2\mu}),$$
$$\mathcal{H}_2:=\{r(\hat{\xi})\hat{\xi}:\, \xi\in \mathcal{H}_1,\, w_u(\hat{\xi})+\mu< r(\hat{\xi})<w_u(\hat{\xi})+2\mu\},$$
and $\mathcal{H}_3:=\mathcal{H}_1\setminus \mathcal{H}_2$.
Notice that 
\be\label{eq:H3}
\mathcal{H}_3=\Big(ce +\underset{\tau>0}{\cup} B_{\tau(c_{uv}+\varepsilon)}(\tau ce)\Big)\cap \{r(\hat{\xi})\hat{\xi}:\, \xi\in \mathcal{H}_1,\, 0\le r(\hat{\xi})\le w_u(\hat{\xi})+\mu\}.
\ee
Since $\overline{\mathcal{H}_2}$ is a compact subset of $\W_v\setminus \overline{\W_u}$ and by Lemma~\ref{lemma:eball-v}, one has
\be\label{H2}
\lim_{t\to +\infty}\sup_{x\in \mathcal{H}_2}\left \{ |u(t, xt)| + |v(t, xt)-1|\right \} =0.
\ee
For any $\delta>0$, there is $T>0$ sufficiently large such that
\be\label{H2-uv}
u(T+s,x)\le \delta,\ v(T+s, x)\ge 1-\delta, \hbox{ for $s\ge 0$ and } x\in (T+s)\mathcal{H}_2.
\ee

Consider the set
$$\mathcal{H}:=\mathcal{H}_1\cup \mathcal{H}_2.$$
Let
\be\label{eq:T'}
T':=\frac{\mu}{\sup_{\xi\in\mathcal{\mathcal{H}}_2}w_u(\hat{\xi})+c_{uv}+\mu+\varepsilon}T.
\ee
We claim that
\be\label{eq:H3H}
(T+T')\mathcal{H}_3+B_{T'(c_{uv}+\varepsilon)}\subset T\mathcal{H}.
\ee 
In fact, for any $s\in [0,T']$,
\be\label{eq:T'-s}
\begin{aligned} 
&(T+T')\Big(ce +\underset{\tau>0}{\cup} B_{\tau(c_{uv}+\varepsilon)}(\tau ce)\Big)+B_{T'(c_{uv}+\varepsilon)}\\
=&(T+T')ce +\underset{\tau>0}{\cup} B_{\tau(c_{uv}+\varepsilon)}(\tau ce)+B_{T'(c_{uv}+\varepsilon)}\\
=& Tce + \underset{\tau>0}{\cup} B_{(\tau+T')(c_{uv}+\varepsilon)}((\tau+T')ce)\\
\subset & Tce + \underset{\tau>0}{\cup} B_{\tau(c_{uv}+\varepsilon)}(\tau ce).
\end{aligned}
\ee
On the other hand, by $w_u(e)\ge c_u>c_{uv}$ and $\mu>\varepsilon$ and \eqref{eq:T'}, $(T+T')(w_u(\hat{\xi})+\mu)+T'(c_{uv}+\varepsilon)=T(w_u(\hat{\xi})+2\mu)-T\mu+T'(w_u(\hat{\xi})+\mu+c_{uv}+\varepsilon)\le T(w_u(\hat{\xi})+2\mu)$ 
for any $s\in [0,T']$ and $\xi\in\mathcal{H}_1$. By \eqref{eq:H3}, \eqref{eq:T'-s} and $\mathcal{H}=\mathcal{H}_1\cup \mathcal{H}_2$, we have \eqref{eq:H3H}.

By \eqref{eq:H3H}, it follows that
\be
\begin{aligned} 
&(T+T')\mathcal{H}_3\cap \{r(\hat{\xi})\hat{\xi}:\, \xi\in \mathcal{H}_1,\, r(\hat{\xi})>T(w_u(\hat{\xi})+\mu)+T'(c_{uv}+\varepsilon)\}+B_{T'(c_{uv}+\varepsilon)}\\
\subset & ((T+T')\mathcal{H}_3+B_{T'(c_{uv}+\varepsilon)})\cap \{r(\hat{\xi})\hat{\xi}:\, \xi\in \mathcal{H}_1,\, r(\hat{\xi})>T(w_u(\hat{\xi})+\mu)\}\\
\subset & T\mathcal{H}_2.
\end{aligned}
\ee
Let $R_{\varepsilon/2}>0$, $\delta>0$ and $T_{\varepsilon/2}$ be given by Lemma~\ref{lem1}.
By taking $T$ large enough such that $\varepsilon T'/2\ge R_{\varepsilon/2}$, one has that for any $x_0\in (T+T')\mathcal{H}_3 \cap \{r(\hat{\xi})\hat{\xi}:\, \xi\in \mathcal{H}_1,\, r(\hat{\xi})>T(w_u(\hat{\xi})+\mu)+T'(c_{uv}+\varepsilon)\}$,
$$u(T,x)\le \delta,\ v(T,x)\ge 1-\delta, \hbox{ for } x\in B_{T'(c_{uv}+\varepsilon/2)+R_{\varepsilon/2}}(x_0).$$
By Lemma~\ref{lem1}, it follows that
$$u(T+s,x)\le 2\delta,\ v(T+s,x)\ge 1-2\delta, \hbox{ for $s\in [0,T']$ and } x\in B_{(T'-s)(c_{uv}+\varepsilon/2)}(x_0),$$
 and
$$u(T+T',x_0)\le \delta,\ v(T+T',x_0)\ge 1-\delta,$$
even if it means increasing $T$ such that $T'\ge T_{\varepsilon/2}$.
Consequently, combined with \eqref{H2-uv}, one has
$$u(T+s, x)\le 2\delta,\ v(T+s, x)\ge 1-2\delta,$$
for $s\in [0,T']$ and  $x\in ((T+T')\mathcal{H}_3\cap \{r(\hat{\xi})\hat{\xi}:\, \xi\in \mathcal{H}_1,\, r(\hat{\xi})>T(w_u(\hat{\xi})+\mu)+T'(c_{uv}+\varepsilon)\}+B_{(T'-s)(c_{uv}+\varepsilon/2)})\cup ((T+s)\mathcal{H}_2)$ and
$$u(T+T',x)\le \delta,\ v(T+T',x)\ge 1-\delta,$$
for $x\in (T+T')\mathcal{H}\cap \{r(\hat{\xi})\hat{\xi}:\, \xi\in \mathcal{H}_1,\, r(\hat{\xi})>T(w_u(\hat{\xi})+\mu)+T'(c_{uv}+\varepsilon)\}$.

As \eqref{eq:H3H},  one has $(T+2T')\mathcal{H}_3+B_{T'(c_{uv}+\varepsilon)}\subset (T+T')\mathcal{H}$ and
\be
\begin{aligned} 
&(T+2T')\mathcal{H}_3\cap \{r(\hat{\xi})\hat{\xi}:\, \xi\in \mathcal{H}_1,\, r(\hat{\xi})>T(w_u(\hat{\xi})+\mu)+2T'(c_{uv}+\varepsilon)\}+B_{T'(c_{uv}+\varepsilon)}\\
\subset & ((T+2T')\mathcal{H}_3+B_{T'(c_{uv}+\varepsilon)})\cap \{r(\hat{\xi})\hat{\xi}:\, \xi\in \mathcal{H}_1,\, r(\hat{\xi})>T(w_u(\hat{\xi})+\mu)+T'(c_{uv}+\varepsilon)\}\\
\subset & (T+T')\mathcal{H}\cap \{r(\hat{\xi})\hat{\xi}:\, \xi\in \mathcal{H}_1,\, r(\hat{\xi})>T(w_u(\hat{\xi})+\mu)+T'(c_{uv}+\varepsilon)\}.
\end{aligned}
\ee
Then, one has that for any $x_0\in (T+2T')\mathcal{H}_3 \cap \{r(\hat{\xi})\hat{\xi}:\, \xi\in \mathcal{H}_1,\, r(\hat{\xi})>T(w_u(\hat{\xi})+\mu)+2T'(c_{uv}+\varepsilon)\}$,
$$u(T+T',x)\le \delta,\ v(T+T',x)\ge 1-\delta, \hbox{ for } x\in B_{T'(c_{uv}+\varepsilon/2)+R_{\varepsilon/2}}(x_0).$$
By Lemma~\ref{lem1}, it follows that
$$u(T+T'+s,x)\le 2\delta,\ v(T+T'+s,x)\ge 1-2\delta, \hbox{ for $s\in [0,T']$ and } x\in B_{(T'-s)(c_{uv}+\varepsilon/2)}(x_0),$$
 and
$$u(T+2T',x_0)\le \delta,\ v(T+2T',x_0)\ge 1-\delta.$$
Consequently, combined with \eqref{H2-uv}, one has
$$u(T+s, x)\le 2\delta,\ v(T+s, x)\ge 1-2\delta,$$
for $s\in [0,T']$ and  $x\in ((T+2T')\mathcal{H}_3\cap \{r(\hat{\xi})\hat{\xi}:\, \xi\in \mathcal{H}_1,\, r(\hat{\xi})>T(w_u(\hat{\xi})+\mu)+2T'(c_{uv}+\varepsilon)\}+B_{(T'-s)(c_{uv}+\varepsilon/2)})\cup ((T+T'+s)\mathcal{H}_2)$ and
$$u(T+T',x)\le \delta,\ v(T+T',x)\ge 1-\delta,$$
for $x\in (T+2T')\mathcal{H}\cap \{r(\hat{\xi})\hat{\xi}:\, \xi\in \mathcal{H}_1,\, r(\hat{\xi})>T(w_u(\hat{\xi})+\mu)+2T'(c_{uv}+\varepsilon)\}$.

By induction,  for any integer $m$, one has
$$u(T+s,x)\le 2\delta,\ v(T+s,x)\ge 1-2\delta,$$
for $x\in ((T+mT')\mathcal{H}_3\cap \{r(\hat{\xi})\hat{\xi}:\, \xi\in \mathcal{H}_1,\, r(\hat{\xi})>T(w_u(\hat{\xi})+\mu)+mT'(c_{uv}+\varepsilon)\}+B_{(mT'-s)(c_{uv}+\varepsilon)})\cup ((T+s)\mathcal{H}_2)$ and $s\in [(m-1)T',mT']$ and
$$u(T+mT',x)\le \delta,\ v(T+mT',x)\ge 1-\delta,$$
for $x\in (T+mT')\mathcal{H}\cap \{r(\hat{\xi})\hat{\xi}:\, \xi\in \mathcal{H}_1,\, r(\hat{\xi})>T(w_u(\hat{\xi})+\mu)+mT'(c_{uv}+\varepsilon)\}$.

Since $ce\not\in \mathcal{S}_u+B_{2\varepsilon}$ and by the definition of $\mathcal{S}_u$, one has that $c\ge c_{uv}+2\varepsilon$. Since $\mathcal{H}_3\subset ce +\underset{\tau>0}{\cup} B_{\tau(c_{uv}+\varepsilon)}(\tau ce)$, one gets that $(T+mT')\mathcal{H}_3\subset\{r(\hat{\xi})\hat{\xi}:\, \xi\in \mathcal{H}_1,\, r(\hat{\xi})>T(w_u(\hat{\xi})+\mu)+mT'(c_{uv}+\varepsilon)\}$ for large $m$. Therefore, it follows that
\be\label{Q-H3}
u(T+s,x)\le 2\delta,\ v(T+s,x)\ge 1-2\delta,
\ee
for $x\in ((T+mT')\mathcal{H}_3+B_{(mT'-s)(c_{uv}+\varepsilon)})\cup ((T+s)\mathcal{H}_2)$ and $s\in [(m-1)T',mT']$ with large $m$.

For any $c'e\in \mathcal{S}_v$ such $c'>w_u(e)$, the conclusion follows immediately from Lemma~\ref{lemma:eball-v}. For any $c'e\in \mathcal{S}_v$ such that $c'\le w_u(e)$, it is clear that there is $c<c'$ such that $ce \in \W_u\setminus \overline{\mathcal{S}_u}$ and $c'e\in \hbox{int}(\mathcal{H}_3)$. Then, there is $\varepsilon'>0$ such that $B_{2\varepsilon'}(c'e)\subset \mathcal{H}_3$. For any $s\in [(m-1)T',mT']$, one has
$$(T+mT')B_{2\varepsilon'}(c'e)=(T+s)c'e+(mT'-s)c'e +B_{\varepsilon'(T+s)+\varepsilon'(T+2mT'-s)}\subset (T+mT')\mathcal{H}_3.$$
By $\varepsilon'(T+2mT'-s)\rightarrow +\infty$ as $m\rightarrow +\infty$ and $(mT'-s)c'\le T'c'$, it follows that $(T+s)B_{\varepsilon'}(c'e)\subset (T+mT')\mathcal{H}_3$ for $s\in [(m-1)T',mT']$ with large $m$. By \eqref{Q-H3}, it finally has 
$$u(T+s,x)\le 2\delta,\ v(T+s,x)\ge 1-2\delta,$$
for $x\in (T+s)B_{\varepsilon'}(c'e)$ and $s\in [(m-1)T',mT']$ with large $m$.
The conclusion follows from the arbitrariness of $\delta$.
\end{proof}

Now, we are ready to prove Theorem~\ref{Th2}.

\begin{proof}[Proof of Theorem~\ref{Th2}]
The set $\W_{uv}$ is a spreading subset of $u$ and the set $\W_v\setminus \overline{\W_{uv}}$ is a spreading superset of $v$ by Lemma~\ref{sycuv}. For any $ce \in\mathcal{S}_u$ such that $ce\in \cup_{0<\tau<1} B_{\tau c_{uv}}((1-\tau)w_u(\xi)\xi)$ for some $\xi\in\mathbb{S}^{N-1}$ such that $w_u(\xi)>w_v(\xi)$, it follows from the openness of $\cup_{0<\tau<1} B_{\tau c_{uv}}((1-\tau)w_u(\xi)\xi)$ that there is $c_{uv}<c'<w_u(\xi)$ and $\varepsilon>0$ such that $ce\in \cup_{0<\tau<1} B_{\tau (c_{uv}-\varepsilon)}((1-\tau)c'\xi)$. Then, Lemma~\ref{Su-cone} implies that there is $\varepsilon'>0$ such that 
$$\lim_{t\to +\infty}\sup_{x\in B_{\varepsilon'}(ce)}\left \{ |u(t, xt)| + |v(t, xt)-1|\right \} =0.$$
Therefore, $\mathcal{S}_u$ is a spreading subset of $u$ and $\R^N\setminus \overline{\mathcal{S}_u}$ is a spreading superset of $v$. For any $ce\in \R^N\setminus \overline{\mathcal{S}_u}$, it follows from the definition of $\mathcal{S}_u$ that either $ce\in \R^N\setminus \overline{\W_u}$ or $ce\in \W_v\setminus \overline{\mathcal{S}_u}$. Then, Lemma~\ref{uv0} and Lemma~\ref{Sv-cone} imply that $\mathcal{S}_u$ is a spreading superset of $u$. Consequently, $\mathcal{S}_u$ is the spreading set of $u$. On the other hand, $\W_v$ is a spreading superset of $v$ by  Lemma~\ref{uv0} and hence, $\W_v\setminus\overline{\mathcal{S}_u}=\mathcal{S}_v$ is a spreading superset of $v$. By Lemma~\ref{Sv-cone}, we can also conclude that $\mathcal{S}_v$ is the spreading set of $v$.

From the definition of $\mathcal{S}_u$, the set $\mathcal{S}_u$ is star-shaped with respect to $0$. So,  the existence of $s_u(e)$ is clear. Then, we prove the formula \eqref{F-su} of $s_u(e)$ and its properties.

First notice the following fact. For any $c>c_{uv}$ and $\xi\in\mathbb{S}^{N-1}$,
$$re\in \underset{0<\tau<1}{\cup}B_{\tau c_{uv}}((1-\tau) c \xi),$$
if and only if
\begin{eqnarray}\label{F-r}
\left\{\begin{array}{lll}
0\le r<c_{uv},&& \hbox{if } \xi\cdot e\le \frac{c_{uv}}{c},\\
0\le r<\frac{c_{uv}c}{\sqrt{1-(\xi\cdot e)^2}\sqrt{c^2+c_{uv}^2}+c_{uv}\xi\cdot e},&& \hbox{if } \xi\cdot e> \frac{c_{uv}}{c},
\end{array}
\right.
\end{eqnarray}
Since $c>c_{uv}$, one can easily show that
\be\label{increasingc}
\frac{c_{uv}c}{\sqrt{1-(\xi\cdot e)^2}\sqrt{c^2+c_{uv}^2}+c_{uv}\xi\cdot e} \hbox{ is increasing in $c$ and $\xi\cdot e$}.
\ee
Then, for any $c'\xi'\in \underset{0<\tau<1}{\cup}B_{\tau c_{uv}}((1-\tau) c \xi)$ with $c'>c_{uv}$, one has that
\be\label{continuity}
\underset{0<\tau<1}{\cup}B_{\tau c_{uv}}((1-\tau) c' \xi')\subset \underset{0<\tau<1}{\cup}B_{\tau c_{uv}}((1-\tau) c \xi).
\ee

On the other hand, by $\W_u=\R^+\mathcal{U}(U)+B_{c_u}$, we claim that
\be\label{claim-Wu}
\W_u=\underset{e\in\mathbb{S}^{N-1}}{\cup}\underset{0<\tau<1}{\cup}B_{\tau c_{u}}((1-\tau) w_u(e) e),
\ee
with the convention that $\cup_{0<\tau<1} B_{\tau c_{u}}((1-\tau)w_u(e)e)=\R^+ e +B_{c_{u}}$ if $w_u(e)=+\infty$. If $\mathcal{U}(U)=\emptyset$, then $w_u(e)=c_{u}$ and
$$\underset{e\in\mathbb{S}^{N-1}}{\cup}\underset{0<\tau<1}{\cup}B_{\tau c_{u}}((1-\tau) w_u(e) e)=B_{c_u}=\W_u.$$
If $\mathcal{U}(U)\neq\emptyset$, it is easy to find that 
$$\W_u\subset \underset{e\in\mathbb{S}^{N-1}}{\cup}\underset{0<\tau<1}{\cup}B_{\tau c_{u}}((1-\tau) w_u(e) e).$$
For $e\in\mathcal{U}(U)$, that is, $w_u(e)=+\infty$, it follows that 
$$\underset{0<\tau<1}{\cup}B_{\tau c_{u}}((1-\tau) w_u(e) e)=\R^+ e +B_{c_{u}}\subset \W_u.$$
For $e\not\in\mathcal{U}(U)$, that is, $w_u(e)<+\infty$, there is $\xi\in \mathcal{U}(U)$ such that 
$$\underset{0<\tau<1}{\cup}B_{\tau c_{u}}((1-\tau) w_u(e) e)\subset \R^+\xi+B_{c_u}\subset \W_u.$$
Thus, 
$$\underset{e\in\mathbb{S}^{N-1}}{\cup}\underset{0<\tau<1}{\cup}B_{\tau c_{u}}((1-\tau) w_u(e) e)\subset \W_u.$$
Then, \eqref{claim-Wu} follows.

By \eqref{claim-Wu}, one has that
$$\underset{e\in\mathbb{S}^{N-1}}{\cup}\underset{0<\tau<1}{\cup}B_{\tau c_{uv}}((1-\tau) w_u(e) e)\subset \W_u.$$
The reversed inclusion is obvious. So, 
\be\label{claim2-Wu}
\W_u=\underset{e\in\mathbb{S}^{N-1}}{\cup}\underset{0<\tau<1}{\cup}B_{\tau c_{uv}}((1-\tau) w_u(e) e).
\ee
By \eqref{F-r}, this implies that
\be\label{max-su}
w_u(e)\ge \frac{c_{uv} w_u(e)}{\sqrt{1-(\xi\cdot e)^2}\sqrt{w_u^2(\xi)+c_{uv}^2}+c_{uv}\xi\cdot e},
\ee
for any $\xi\in\mathbb{S}^{N-1}$ such that $\xi\cdot e>c_{uv}/w_u(\xi)$, and the inequality is strict if $\xi\neq e$.

By the definition of $\mathcal{S}_u$ and \eqref{claim2-Wu}, it is clear that
\be\label{WuvSuWu}
\W_{uv}\subset \mathcal{S}_u\subset \W_u.
\ee

For $e\in\mathcal{U}(U)$, one has that $\R^+e\subset \W_{uv}\subset \mathcal{S}_u$ which implies that $s_u(e)=+\infty$.

For $e\not\in\mathcal{U}(U)$ such that $w_u(e)\ge w_v(e)$, it follows from \eqref{WuvSuWu} that $re\in \mathcal{S}_u$ if and only if $0\le r\le w_u(e)$. Meantime, by \eqref{max-su} and \eqref{increasingc} , we have that
\be
\begin{aligned} 
s_u(e)=&\sup_{\substack{\xi\in\mathbb{S}^{N-1},\\ w_u(\xi)\ge w_v(\xi), \\ \xi\cdot e> c_{uv}/w_u(\xi)}}\sup_{0\le c<w_u(\xi)}\frac{c c_{uv}}{\sqrt{1-(\xi\cdot e)^2} \sqrt{c^2 +c_{uv}^2}+c_{uv}\xi\cdot e }\\
=&\sup_{\substack{\xi\in\mathbb{S}^{N-1},\\ w_u(\xi)\ge w_v(\xi), \\ \xi\cdot e> c_{uv}/w_u(\xi)}} \frac{w_u(\xi) c_{uv}}{\sqrt{1-(\xi\cdot e)^2} \sqrt{w^2_u(\xi) +c_{uv}^2}+c_{uv}\xi\cdot e }=w_u(e).
\end{aligned}
\ee
For $e\not\in\mathcal{U}(U)$ such that $w_u(e)< w_v(e)$, it follows from \eqref{F-r} that $re\in \mathcal{S}_u$ if and only if $0\le r<s_u(e)$. By \eqref{max-su} and \eqref{WuvSuWu}, $w_{uv}(e)\le s_u(e)<w_u(e)$.

The continuity of $s_u(e)$ follows from \eqref{continuity} and the definition of $\mathcal{S}_u$.
\end{proof}
\vskip 0.5cm

\noindent
\textbf{Data availability.}
No data was used for the research described in the article.
\vskip 0.3cm

\noindent
\textbf{Conflict of interest.}
All authors affirms that there is no conflict of interest.

\vskip 0.3cm
\noindent
\textbf{Acknowledgement.} H. Guo was supported by the fundamental research funds for the central universities and NSF of China (12471201).



\end{document}